\documentclass[11pt, one sided]{amsart}
\RequirePackage[colorlinks,citecolor=blue,urlcolor=blue]{hyperref}

\usepackage{latexsym,amssymb,amsmath,amsfonts,amsthm}
\usepackage{amsthm,amsmath}
\usepackage{float}
\usepackage{enumerate}
\usepackage{comment}
\usepackage[margin=3cm]{geometry}
\usepackage{bbm}
\usepackage{subfigure}
\usepackage{graphicx}
\usepackage[toc,page]{appendix}
\usepackage{multirow}
\usepackage{amsfonts}
\usepackage[all]{xy}
\usepackage{caption}
 \usepackage{geometry}                
\usepackage{graphicx}
\usepackage{amssymb}
\usepackage{epstopdf}
\usepackage{color}
\usepackage{bbm, dsfont}
\usepackage{hyperref}
\usepackage[utf8]{inputenc}
\usepackage{amssymb,amsthm,amsmath,amssymb,wrapfig,dsfont}
\usepackage[dvipsnames]{xcolor}
\definecolor{myred}{RGB}{251,154,133}
\definecolor{myblue}{RGB}{153,206,227}
\definecolor{mylightblue}{RGB}{0, 150, 255}
\definecolor{mygreen}{RGB}{32, 210, 64}
\definecolor{mygray}{RGB}{220, 220, 220}

\usepackage{tikz}
\usetikzlibrary{decorations.pathmorphing}
\tikzset{snake it/.style={decorate, decoration=snake}}
\usetikzlibrary{shapes.geometric,positioning,decorations.pathreplacing} 
\usepackage{pgfplots}

\newtheorem{theorem}{Theorem}

\newtheorem{lemma}{Lemma}
\newtheorem{remark}{Remark}[section]

\newtheorem{corollary}{Corollary}[section]

\newtheorem{open_problem}{Open Problem}

\def\beq{ \begin{equation} }
\def\eeq{ \end{equation} }

\def\ep{\varepsilon}

\def\square{\vcenter{\vbox{\hrule height .4pt
  \hbox{\vrule width .4pt height 5pt \kern 5pt
        \vrule width .4pt} \hrule height .4pt}}}

\newcommand{\BN}{{\mathbb{N}}}

\newcommand{\BZ}{{\mathbb{Z}}}

\newcommand{\prob}{{\bf P}}

\newcommand{\bae}{\begin{equation}\begin{aligned}}
\newcommand{\eae}{\end{aligned}\end{equation}}

\DeclareFontFamily{OML}{rsfs}{\skewchar\font'177}
\DeclareFontShape{OML}{rsfs}{m}{n}{ <5> <6> rsfs5 <7> <8> <9>
rsfs7 <10> <10.95> <12> <14.4> <17.28> <20.74> <24.88> rsfs10 }{}
\DeclareMathAlphabet{\mathfs}{OML}{rsfs}{m}{n}

\pgfplotsset{compat=1.18}
\newcommand{\abs}[1]{\left| #1 \right|}
\newcommand{\rb}[1]{\left( #1 \right)}
\newcommand{\bb}[1]{\left[ #1 \right]}
\newcommand{\cb}[1]{\left\{ #1 \right\}}

\newcommand{\ncon}[3]{#1 \overset{#3}{\nleftrightarrow} #2}
\newcommand{\con}[3]{#1 \overset{#3}{\leftrightarrow} #2}

\newcommand{\oov}[1]{\frac{1}{#1}}
\newcommand{\sumn}[1]{\sum_{i=1}^{n}{#1}}

\DeclareMathOperator*{\argmax}{arg\,max} 
\DeclareMathOperator*{\argmin}{arg\,min}

\newcommand{\Prob}[2][]{\mathbb{P}_{#1}\left(#2\right)}

\begin{document}

\title{On one-dimensional Cluster cluster model}

\author{Noam Berger}
\address[Noam Berger]{Department of mathematics, TUM}
\urladdr{https://www.math.cit.tum.de/en/probability/people/berger/}
\email{noam.berger@tum.de }

\author{Eviatar B. Procaccia}
\address[Eviatar B. Procaccia]{Faculty of data and decision sciences, Technion}
\urladdr{https://procaccia.net.technion.ac.il}
\email{eviatarp@technion.ac.il}

\author{Daniel Sharon}
\address[Daniel Sharon]{Faculty of data and decision sciences, Technion}
\email{danielsharon@campus.technion.ac.il}

\begin{abstract}
The Cluster-cluster model was introduced by Meakin et al in 1984. 
Each $x\in \mathds{Z}^d$ starts with a cluster of size 1 with probability $p \in (0,1]$ independently. 
Each cluster $\mathfs{C}$ performs a continuous-time SRW with rate $\abs{\mathfs{C}}^{-\alpha}$. 
If it attempts to move to a vertex occupied by another cluster, it does not move, and instead the two clusters connect via a new edge.

Focusing on dimension $d=1$, we show that for $\alpha>-2$, at time $t$, the cluster size is of order $t^\frac{1}{\alpha + 2}$, and for $\alpha < -2$ we get an infinite cluster in finite time a.s. 
Additionally, for $\alpha = 0$ we show convergence in distribution of the scaling limit.

\end{abstract}
\maketitle
\tableofcontents

\section{Introduction}


The Cluster-cluster model was introduced by Meakin et al in 1984 \cite{meakin1984diffusion} and and studied in the physics literature \cite{meakin1984effects, meakin1985dynamic,rajesh2024exact}. The model was introduced as a more realistic model than Diffusion Limited Aggregation (DLA) for colloidal systems exhibiting flocculation, such as blood coagulation \cite{leyvraz2003scaling}, cloud formation \cite{lushnikov2006gelation}, aerosol dynamics \cite{lushnikov1978coagulation} and protein aggregation \cite{frieden2007protein}.
In this paper we present the first rigorous mathematical treatment of the model. 

Denote $\mathcal{E}(\BZ^d)$ the set of edges of $\BZ^d$. Abbreviate $B_r(x):=\{y\in\BZ^d: \|y-x\|_\infty\le r\}$.
We are interested in a stochastic process $$\{\mathfs{A}_t\}_{t\ge 0}=\{\mathfs{V}_t\times \mathfs{E}_t\}_{t\ge 0}$$ over $\{0,1\}^{\BZ^d}\times \{0,1\}^{\mathcal{E}(\BZ^d)}$. 

A Cluster $\mathfs{C}$ is a maximal set of open vertices which are connected by edge paths together with all open edges between vertices. We denote by $\partial \mathfs{C}$ the $\BZ^d$-edge boundary of the vertices in $\mathfs{C}$, and by $\partial_{\mathfs{E}_t}\mathfs{C}$ the edge boundary in the sub-graph of open vertices and edges given by $\mathfs{A}_t$. 

$\{\mathfs{A}_t^\alpha\}_{t\ge 0}$ is a Cluster cluster model with parameter $\alpha$ if it posseses the following properties:
\begin{enumerate}
\item $\mathfs{A}_0$ is samples from a product measure with marginal distribution $\prob(\mathfs{A}_0(v)=1)=p$. The edge set is empty at initiation, $\mathfs{E}_0=\emptyset$. 
\item Every cluster $\mathfs{C}$ has a Poisson clock at rate $|\mathfs{C}|^{-\alpha}$.
\item 
    \begin{itemize} 
        \item ($d\ge 2$): Upon ringing a cluster attempts to preform a random walk move (all vertices and edges of the cluster move together). If destination is vacant, then cluster preforms the move. Otherwise, choose independently, uniformly, an edge among the edges connecting to clusters inhibiting the move (and thus connecting two clusters).
        \item ($d=1$): Whenever two clusters are of distance $1$, clusters connect via the edge between them. 
    \end{itemize}
\end{enumerate}

\begin{figure}
    \centering
    \includegraphics[width=0.5\linewidth]{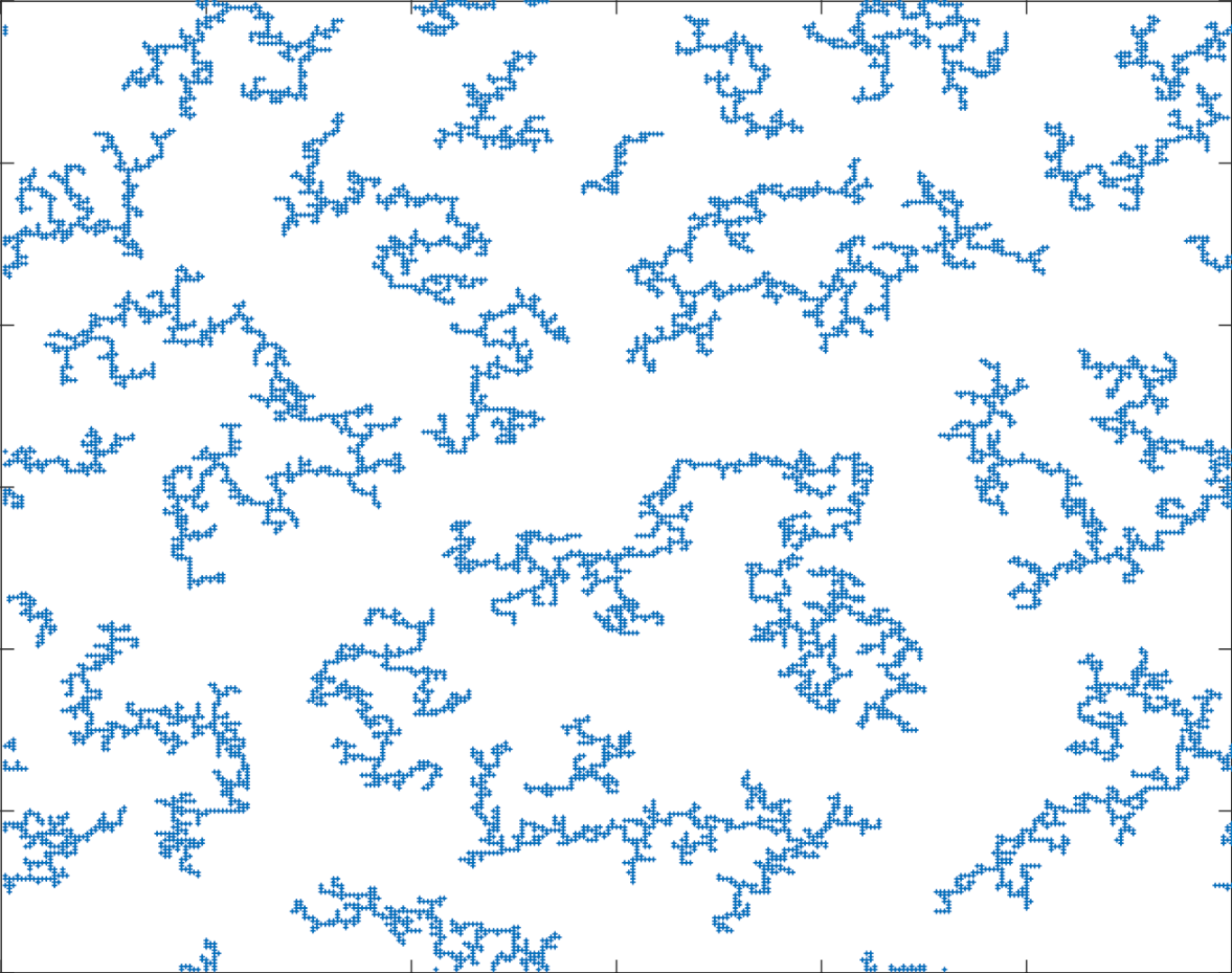}
    \caption{Cluster-cluster model on a torus with $\alpha=1$.}
    \label{fig:enter-label}
\end{figure}

\begin{figure}[ht]
    \centering

    \begin{minipage}{0.24\textwidth}
        \centering
        \begin{tikzpicture}[yscale=0.5,xscale=0.5]
            \draw[step=1cm,blue,thin, dotted] (0,0) grid (6,6);
            \node at (1,0) [circle,fill=black,inner sep=1pt]{};
            \draw [black,very thick] (1,0) -- (1,1);
            \node at (1,1) [circle,fill=black,inner sep=1pt]{};
            \draw [black,very thick] (1,1) -- (1,2);
            \node at (1,2) [circle,fill=black,inner sep=1pt]{};
            \draw [black,very thick] (1,2) -- (1,3);
            \node at (1,3) [circle,fill=black,inner sep=1pt]{};
            \draw [black,very thick] (1,2) -- (2,2);
            \node at (2,2) [circle,fill=black,inner sep=1pt]{};
            \draw [black,very thick] (2,2) -- (2,3);
            \node at (2,3) [circle,fill=black,inner sep=1pt]{};
            \node at (2,0) [circle,fill=black,inner sep=1pt]{};
            \draw [black,very thick] (2,0) -- (2,1);
            \node at (2,1) [circle,fill=black,inner sep=1pt]{};
            \draw [black,very thick] (2,1) -- (3,1);
            \node at (3,1) [circle,fill=black,inner sep=1pt]{};
            \node at (6,2) [circle,fill=black,inner sep=1pt]{};
            \node at (5,0) [circle,fill=black,inner sep=1pt]{};
            \node at (5,1) [circle,fill=black,inner sep=1pt]{};
            \draw [black,very thick] (5,0) -- (5,1);
            \draw [black,very thick] (5,1) -- (5,2);
            \node at (5,2) [circle,fill=black,inner sep=1pt]{};
            \draw [black,very thick] (5,2) -- (6,2);
            \draw [black,very thick] (5,2) -- (5,3);
            \node at (5,3) [circle,fill=black,inner sep=1pt]{};
            \draw [black,very thick] (5,3) -- (5,4);
            \node at (5,4) [circle,fill=black,inner sep=1pt]{};
            \draw [black,very thick] (5,4) -- (6,4);
            \node at (6,4) [circle,fill=black,inner sep=1pt]{};
            \draw [black,very thick] (0,4) -- (1,4);
            \node at (1,4) [circle,fill=black,inner sep=1pt]{};
            \draw [black,very thick] (1,4) -- (2,4);
            \node at (2,4) [circle,fill=black,inner sep=1pt]{};
            \draw [black,very thick] (2,4) -- (3,4);
            \node at (3,4) [circle,fill=black,inner sep=1pt]{};
        \end{tikzpicture}
    \end{minipage}%
    \hfill
    \begin{minipage}{0.24\textwidth}
        \centering
        \begin{tikzpicture}[yscale=0.5,xscale=0.5]
            \draw[step=1cm,blue,thin, dotted] (0,0) grid (6,6);
            \node at (1,0) [circle,fill=black,inner sep=1pt]{};
            \draw [black,very thick] (1,0) -- (1,1);
            \node at (1,1) [circle,fill=black,inner sep=1pt]{};
            \draw [black,very thick] (1,1) -- (1,2);
            \node at (1,2) [circle,fill=black,inner sep=1pt]{};
            \draw [black,very thick] (1,2) -- (1,3);
            \node at (1,3) [circle,fill=black,inner sep=1pt]{};
            \draw [black,very thick] (1,2) -- (2,2);
            \node at (2,2) [circle,fill=black,inner sep=1pt]{};
            \draw [black,very thick] (2,2) -- (2,3);
            \node at (2,3) [circle,fill=black,inner sep=1pt]{};
            \node at (2,0) [circle,fill=black,inner sep=1pt]{};
            \draw [black,very thick] (2,0) -- (2,1);
            \node at (2,1) [circle,fill=black,inner sep=1pt]{};
            \draw [black,very thick] (2,1) -- (3,1);
            \node at (3,1) [circle,fill=black,inner sep=1pt]{};
            \node at (4,0) [circle,fill=black,inner sep=1pt]{};
            \node at (4,1) [circle,fill=black,inner sep=1pt]{};
            \draw [black,very thick] (4,0) -- (4,1);
            \draw [black,very thick] (4,1) -- (4,2);
            \node at (4,2) [circle,fill=black,inner sep=1pt]{};
            \draw [black,very thick] (4,2) -- (5,2);
            \node at (5,2) [circle,fill=black,inner sep=1pt]{};
            \draw [black,very thick] (4,2) -- (4,3);
            \node at (4,3) [circle,fill=black,inner sep=1pt]{};
            \draw [black,very thick] (4,3) -- (4,4);
            \node at (4,4) [circle,fill=black,inner sep=1pt]{};
            \draw [black,very thick] (4,4) -- (5,4);
            \node at (5,4) [circle,fill=black,inner sep=1pt]{};
            \draw [black,very thick] (0,4) -- (1,4);
            \node at (1,4) [circle,fill=black,inner sep=1pt]{};
            \draw [black,very thick] (1,4) -- (2,4);
            \node at (2,4) [circle,fill=black,inner sep=1pt]{};
            \draw [black,very thick] (2,4) -- (3,4);
            \node at (3,4) [circle,fill=black,inner sep=1pt]{};
        \end{tikzpicture}
    \end{minipage}%
    \hfill
    \begin{minipage}{0.24\textwidth}
        \centering
        \begin{tikzpicture}[yscale=0.5,xscale=0.5]
            \draw[step=1cm,blue,thin, dotted] (0,0) grid (6,6);
            \node at (1,0) [circle,fill=black,inner sep=1pt]{};
            \draw [black,very thick] (1,0) -- (1,1);
            \node at (1,1) [circle,fill=black,inner sep=1pt]{};
            \draw [black,very thick] (1,1) -- (1,2);
            \node at (1,2) [circle,fill=black,inner sep=1pt]{};
            \draw [black,very thick] (1,2) -- (1,3);
            \node at (1,3) [circle,fill=black,inner sep=1pt]{};
            \draw [black,very thick] (1,2) -- (2,2);
            \node at (2,2) [circle,fill=black,inner sep=1pt]{};
            \draw [black,very thick] (2,2) -- (2,3);
            \node at (2,3) [circle,fill=black,inner sep=1pt]{};
            \node at (2,0) [circle,fill=black,inner sep=1pt]{};
            \draw [black,very thick] (2,0) -- (2,1);
            \node at (2,1) [circle,fill=black,inner sep=1pt]{};
            \draw [black,very thick] (2,1) -- (3,1);
            \node at (3,1) [circle,fill=black,inner sep=1pt]{};
            \node at (4,0) [circle,fill=black,inner sep=1pt]{};
            \node at (4,1) [circle,fill=black,inner sep=1pt]{};
            \draw [black,very thick] (4,0) -- (4,1);
            \draw [black,very thick] (4,1) -- (4,2);
            \node at (4,2) [circle,fill=black,inner sep=1pt]{};
            \draw [black,very thick] (4,2) -- (5,2);
            \node at (5,2) [circle,fill=black,inner sep=1pt]{};
            \draw [black,very thick] (4,2) -- (4,3);
            \node at (4,3) [circle,fill=black,inner sep=1pt]{};
            \draw [black,very thick] (4,3) -- (4,4);
            \node at (4,4) [circle,fill=black,inner sep=1pt]{};
            \draw [black,very thick] (4,4) -- (5,4);
            \node at (5,4) [circle,fill=black,inner sep=1pt]{};
            \draw [black,very thick] (0,4) -- (1,4);
            \node at (1,4) [circle,fill=black,inner sep=1pt]{};
            \draw [black,very thick] (1,4) -- (2,4);
            \node at (2,4) [circle,fill=black,inner sep=1pt]{};
            \draw [black,very thick] (2,4) -- (3,4);
            \node at (3,4) [circle,fill=black,inner sep=1pt]{};
            \draw [red,dotted,very thick] (3,1) -- (4,1);
            \draw [red,dotted,very thick] (3,4) -- (4,4);
        \end{tikzpicture}
    \end{minipage}%
    \hfill
    \begin{minipage}{0.24\textwidth}
        \centering
        \begin{tikzpicture}[yscale=0.5,xscale=0.5]
            \draw[step=1cm,blue,thin, dotted] (0,0) grid (6,6);
            \node at (1,0) [circle,fill=black,inner sep=1pt]{};
            \draw [black,very thick] (1,0) -- (1,1);
            \node at (1,1) [circle,fill=black,inner sep=1pt]{};
            \draw [black,very thick] (1,1) -- (1,2);
            \node at (1,2) [circle,fill=black,inner sep=1pt]{};
            \draw [black,very thick] (1,2) -- (1,3);
            \node at (1,3) [circle,fill=black,inner sep=1pt]{};
            \draw [black,very thick] (1,2) -- (2,2);
            \node at (2,2) [circle,fill=black,inner sep=1pt]{};
            \draw [black,very thick] (2,2) -- (2,3);
            \node at (2,3) [circle,fill=black,inner sep=1pt]{};
            \node at (2,0) [circle,fill=black,inner sep=1pt]{};
            \draw [black,very thick] (2,0) -- (2,1);
            \node at (2,1) [circle,fill=black,inner sep=1pt]{};
            \draw [black,very thick] (2,1) -- (3,1);
            \node at (3,1) [circle,fill=black,inner sep=1pt]{};
            \node at (4,0) [circle,fill=black,inner sep=1pt]{};
            \node at (4,1) [circle,fill=black,inner sep=1pt]{};
            \draw [black,very thick] (4,0) -- (4,1);
            \draw [black,very thick] (4,1) -- (4,2);
            \node at (4,2) [circle,fill=black,inner sep=1pt]{};
            \draw [black,very thick] (4,2) -- (5,2);
            \node at (5,2) [circle,fill=black,inner sep=1pt]{};
            \draw [black,very thick] (4,2) -- (4,3);
            \node at (4,3) [circle,fill=black,inner sep=1pt]{};
            \draw [black,very thick] (4,3) -- (4,4);
            \node at (4,4) [circle,fill=black,inner sep=1pt]{};
            \draw [black,very thick] (4,4) -- (5,4);
            \node at (5,4) [circle,fill=black,inner sep=1pt]{};
            \draw [black,very thick] (0,4) -- (1,4);
            \node at (1,4) [circle,fill=black,inner sep=1pt]{};
            \draw [black,very thick] (1,4) -- (2,4);
            \node at (2,4) [circle,fill=black,inner sep=1pt]{};
            \draw [black,very thick] (2,4) -- (3,4);
            \node at (3,4) [circle,fill=black,inner sep=1pt]{};
            \draw [black,very thick] (3,1) -- (4,1);
        \end{tikzpicture}
    \end{minipage}%
    \caption{Example of move step and connection step}
    \label{fig:all-configs}
\end{figure}
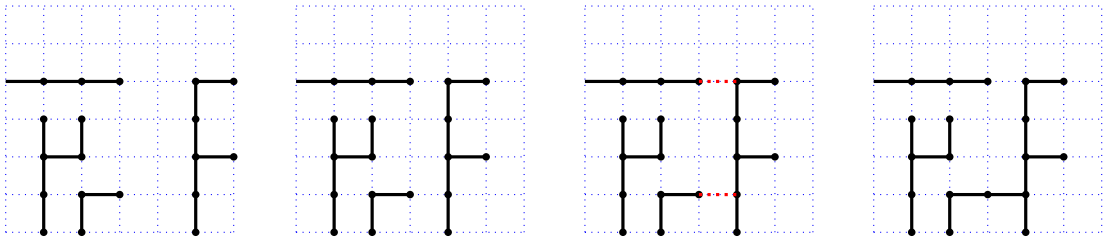
Note that in dimension $1$, the boundaries of clusters are always of size $2$, thus when a cluster moves, it can encounter at most one other cluster. Thus, one does not need a mechanism to choose connections such as in $d>1$.
In this paper, we focus on the one-dimensional variant of the model, and provide sharp growth bounds for the size of a single cluster.

\section{Results and Open Problems}
Denote $\mathfs{C}_0(t)$ the cluster started closest to $0$ at time $t$.
The first result is an exact cluster size limiting distribution for the case of homogeneous speed ($\alpha = 0$).
\begin{theorem}\label{thm:alpha=0}
    Let $d=1, \alpha = 0, 0<p<1$. $\forall x>0$:
    \begin{equation}\label{eq:party10}
        \begin{aligned}
            \lim_{t \to \infty}\Prob{\frac{|\mathfs{C}_0(t)|}{\sqrt{t}} \le x}=2\Phi\rb{\frac{x(\eta-1)}{2}} - \frac{x(\eta-1)}{\sqrt{2\pi}}e^{-\frac{\big(\frac{x}{2}(\eta-1)\big)^2}{2}} - 1
        ,\end{aligned}
    \end{equation}
    where $\Phi$ is the cdf of a $N(0,1)$ and $\eta = p^{-1}$.
\end{theorem}

\begin{remark}
Note that \eqref{eq:party10} means that the density of the limiting distribution is 
\[
{\bf 1}_{x>0}\cdot\frac{\gamma x^2}{\sqrt{2\pi}} e^{-\gamma x^2/2}
\]
for $\gamma = (\eta -1)^2/4$, namely the distribution is a gaussian weighted by $x^2$.
\end{remark}

The second result pertains to the asymptotic growth rate.
\begin{theorem}\label{thm:growthRate1}
    Let $d=1, \alpha \ge 0$, $0<p<1$.
    $\forall \ep>0$, $\exists T(\ep)>0, c(\ep)>0$ s.t $\forall t>T$:
    \begin{equation}
        \Prob{\frac{1}{c}t^{\frac{1}{\alpha+2}} \le \abs{\mathfs{C}_0(t)} \le ct^{\frac{1}{\alpha+2}}} \ge 1-\ep
    .\end{equation}
\end{theorem}

\begin{theorem}\label{thm:growthRate2}
    Let $d=1, -2<\alpha<0$, $0<p<1$.
    $\forall \ep>0, \delta>0$, $\exists T(\ep,\delta)>0, c(\ep,\delta)>0$ s.t $\forall t>T$:
    \begin{equation}
        \Prob{\frac{1}{c}t^{\frac{1}{\alpha+2}-\delta} \le \abs{\mathfs{C}_0(t)} \le ct^{\frac{1}{\alpha+2}}} \ge 1-\ep
    .\end{equation}  
    For $\alpha < -2$:
    \begin{equation}
        \lim_{t \to \infty}{\Prob{\abs{\mathfs{C}_0(t)} < \infty}} = 0
    \end{equation}
\end{theorem}

the following corollary is an easy consequence of the proof of Theorem \ref{thm:growthRate2} using the ergodicity of the model, and the fact that the proof provides positive probability conditions for blow up in any finite time.

\begin{corollary}
    The process is not well defined for $d=1, \alpha < -2$ 
\end{corollary}

We can see that this result for $-2< \alpha < 0$ gives a weaker result than the equivalent theorem for $\alpha > 0$, since we cannot take $c = c_n \to 0$ as $n \to \infty$. However, $\gamma_n \to \frac{1}{\alpha + 2}$ as $n \to \infty$, so we can get arbitrarily close to the correct exponent (given also by the upper bound).

\begin{open_problem}
    Find the behavior for $d=1$ and the critical parameter $\alpha = -2$.
\end{open_problem}

In higher dimensions we conjecture that $\alpha=-1$ is the critical value for blowup. 
\begin{open_problem}
Show that for all $d>1$, and any $\alpha>-1$, all clusters are finite a.s.
\end{open_problem}
\begin{open_problem}
Show that for all $d>1$, and any $\alpha<-1$, there is blowup at finite time a.s.
\end{open_problem}

These problems, with partial sulotions, are discussed in our work in progress \cite{wop}

Note that even for simpler aggregation processes with a single aggregate, not much is currently known about growth rates. For DLA, the best result is still Kesten's \cite{kesten1987long}, where it is proved that DLA with $n$ particles is contained in a ball of radius $n^{2/3}$. Similar results were recently established for the dialectic breakdown model \cite{losev2025long}. The only DLA type model in which we have exact growth bounds is an off-lattice version constructed by a composition of conformal maps called the stationary Hastings-Levitov(0) \cite{berger2022growth,procaccia2021dimension}.  For Multi-particle DLA (MDLA), a model which has a single aggregate but driven by a field of random walkers, much as our model, in high dimension linear growth rate is known in the high intensity regime. For 1d MDLA, exact growth bounds are known, with interesting critical behaviour \cite{elboim2020critical,sidoravicius2017one,sly2020one}. Mean field versions of our model were considered by employing equations resembling the Smoluchowski equations \cite{smoluchowski1918versuch}, as part of the study of Oswald ripening \cite{ostwald1903lehrbuch,voorhees1985theory}. A slightly different approach was taken in \cite{rath2009erdHos}, where the authors consider a {\em mean field forest fire model} which follows the Smoluchowski equations, but where the emerging infinite component instantly breaks down into single particles.

\section{Definitions and notations for $d=1$}
The following discussion considers the case $d=1$, $p\in(0,1)$.
Let $X_i(t)$ denote the position of the $i$'th particle to the right of $0$'th particle, with the $0$'th particle being the particle closest to $0$. Denote by $\mathfs{C}_i(t)$ the connected component of $X_i(t)$ at time $t$ (this does not define clusters uniquely by index). Denote by $|\mathfs{C}_i(t)|$ its size. Denote by $x \overset{t}{\leftrightarrow} y$ the event that $\mathfs{C}_x(t) = \mathfs{C}_y(t)$. 

Note that $X_i(t)$ and $\mathfs{C}_i(t)$ may take non integer $i$ values, in which case we consider $X_{\lceil i \rceil}(t)$ and $ \mathfs{C}_{\lceil i \rceil}(t) $

Define $\Tilde{\tau}(x,y) = \inf\cb{t \ge 0 : x \overset{t}{\leftrightarrow} y}$, and let $S_x^y$ be the number of discrete steps taken by $\mathfs{C}_x$ until time $\Tilde{\tau}(x,y)$. Let $t_i^x$ be the time taken for $\mathfs{C}_x$ to take the $i$'th step. Denote $T_{i,n}^x = \sum_{j=i+1}^{n}{t_j^x}$, $T_n^x = T_{0,n}^x$ and $\tau_x^y = \sum_{i=1}^{S_x^y}{t_i^x} = T_{S_x^y}^x$. Note that $\tau_x^y\neq \tau_y^x$ as they differ in the time of the last step taken to connect the clusters. 

Define $\mathfs{C}_x(t-) := \bigcap_{\ep>0}\bigcup_{t-\ep<s<t}{\mathfs{C}_x(s)}$ to be the cluster right before any change made at exactly time $t$.

From the definition of the step times $t_i^x$, we can see that they follow the law of a the inter arrival times of a Poisson process with non uniform (random) rate $\lambda(t) = |\mathfs{C}_x(t)|^{-\alpha}$.  We can thus define times $\Tilde{t}_i^x$ and $\Tilde{T}_n^x = \sum_{i=1}^{n}{\Tilde{t}_i^x}$
\begin{equation}
    u_0^x(t) = \int_{0}^{t}{\abs{\mathfs{C}_x(s)}^{-\alpha} ds} =\int_{0}^{t}{\lambda (s) ds} \qquad \Tilde{T_n} = u_0^x\rb{T_n}
\end{equation} 
and have that $\Tilde{t_i^x}$ distributed as the inter arrival times of a rate 1 Poisson Process (i.i.d $\Tilde{t_i^x} \sim exp(1)$). This can be seen in
\cite[P.22-23]{DaleyJones2003PointProcesses}, where $N(t) = \cb{\text{\# moves of $\mathfs{C}_x$ up to time t}}$ and we define $\Tilde{N}(t) = N(u^{-1}(t))$. We get that $\Tilde{N}$ is a rate 1 Poisson process with arrival times $\Tilde{T_n}$, Since $\Tilde{N}\rb{\Tilde{T_n}} = N\rb{u^{-1}\rb{u\rb{T_n}}} = N\rb{T_n}$ gives the arrival times of $N$. 
A slightly different and in some cases more useful way to write this, is to consider starting the process right after the $(k-1)$'th step of $\mathfs{C}_x$, which corresponds to time $\sum_{i=1}^{k-1}{t_i^x}$: 
\begin{equation}
    u_\tau^x(t) = \int_{\tau}^{t+\tau}{\abs{\mathfs{C}_x(s)}^{-\alpha} ds} =\int_{\tau}^{t+\tau}{\lambda (s) ds} \qquad \Tilde{T}_{k,n}^x = u_{T_k^x}^x\rb{T_{k,n}^x}
\end{equation}

This characterization gives us the following property. Given $k,n,v\in \BN$ s.t $k<n, x\in \mathds{Z}, \tau \in \mathds{R}_+$ we have the following. 
\begin{lemma}\label{lem:sizeBoundTime}
For $\alpha < 0$:
\begin{equation}\label{slow<0}
    \Prob{T_{k,n}^x \le \tau, \abs{\mathfs{C}_x\rb{T_n^x-}} \le v} \le \Prob{v^\alpha \Tilde{T}_{k,n}^x \le \tau}
\end{equation}
and
\begin{equation}\label{fast<0}
    \Prob{T_{k,n}^x \ge \tau, \abs{\mathfs{C}_x\rb{T_k^x}} \ge v} \le \Prob{v^\alpha \Tilde{T}_{k,n}^x \ge \tau}
.\end{equation}

For $\alpha > 0$:
\begin{equation}\label{slow>0}
    \Prob{T_{k,n}^x \le \tau, \abs{\mathfs{C}_x\rb{T_k^x}} \ge v} \le \Prob{v^\alpha \Tilde{T}_{k,n}^x \le \tau}
\end{equation}
and
\begin{equation}\label{fast>0}
    \Prob{T_{k,n}^x \ge \tau, \abs{\mathfs{C}_x\rb{T_n^x-}} \le v} \le \Prob{v^\alpha \Tilde{T}_{k,n}^x \ge \tau}
.\end{equation}
\end{lemma}

\begin{proof} We will show \eqref{slow<0}, as the other inequalities are done in the same way, with some inequalities reversed. 
Given $\cb{\abs{\mathcal{C}_x\rb{T_n^x-}} \le v}$, by monotonicity of $\abs{\mathcal{C}_x(\cdot)}$, we have that $\abs{\mathcal{C}_x(t)} \le v$ for all $0 \le t < T_n^x$. This gives us, since $\alpha < 0$, that $u_{T_k}^x(t) \le v^{-\alpha} t$ for all $0 \le t \le T_n^x - T_k^x = T_{k,n}^x$. Setting $t = T_{k,n}^x$ we obtain ${T}_{k,n}^x\ge v^\alpha \Tilde{T}_{k,n}^x$. 

\end{proof}

\section{Exact cluster size distribution for $\alpha = 0$}
The following is true for $d=1,\alpha = 0, 0<p<1$. 
\begin{theorem}\label{alpha=0}
    $\forall c>0$:
    \begin{equation}
        \begin{aligned}
            \lim_{t \to \infty}\Prob{\frac{|\mathfs{C}_0(t)|}{\sqrt{t}} \le x}=2\Phi\rb{\frac{x(\eta-1)}{2}} - \frac{x(\eta-1)}{\sqrt{2\pi}}e^{-\frac{\big(\frac{x}{2}(\eta-1)\big)^2}{2}} - 1
        \end{aligned}
    \end{equation}
    where $\Phi$ is the cdf of a $N(0,1)$.
\end{theorem}

\begin{proof}
Let $c \in \mathbb{R}$ and denote $n=\lceil c\sqrt{t} \rceil-1$.
\begin{equation}
    \begin{aligned}
        &\Prob{\abs{\mathfs{C}_0(t)} = n+1}\\ 
        &= \sum_{i=0}^{n}{\Prob{i-n \overset{t}{\leftrightarrow} i, i-n-1\overset{t}{\nleftrightarrow} i-n, i\overset{t}{\nleftrightarrow} i+1}} \\
        &= (n+1)\rb{\Prob{\con{0}{n}{t}, \ncon{-1}{0}{t}, \ncon{n}{n+1}{t}}} \\
        &= (n+1)\rb{\Prob{\con{0}{n}{t}, \ncon{-1}{0}{t}} - \Prob{\con{0}{n}{t}, \ncon{-1}{0}{t}, \con{n}{n+1}{t}}} \\
        &= (n+1)\rb{\Prob{\con{0}{n}{t}, \ncon{-1}{0}{t}} - \Prob{\con{0}{n+1}{t}, \ncon{-1}{0}{t}}} \\
        &= (n+1) \left ( \rb{\Prob{\con{0}{n}{t}} - \Prob{\con{0}{n}{t}, \con{-1}{0}{t}}} \right.  \\
        & \qquad \qquad \quad -  \left. \rb{\Prob{\con{0}{n+1}{t}} - \Prob{\con{0}{n+1}{t}, \con{-1}{0}{t}}}\right) \\
        &= (n+1) \left ( \rb{\Prob{\con{0}{n}{t}} - \Prob{\con{0}{n+1}{t}}} \right.  
         - \left. \rb{\Prob{\con{0}{n+1}{t}} - \Prob{\con{0}{n+2}{t}}}\right) \\
        &= (n+1)\rb{\Delta_{n}-\Delta_{n+1}}
    \end{aligned}
\end{equation}
where we used disjoint partitions and translation invariance, and we denote $\Delta_n = \Prob{\con{0}{n}{t}} - \Prob{\con{0}{n+1}{t}}$.
Now we can calculate the CDF using this:
\begin{equation}\label{CDFcalc}
    \begin{aligned}
        \Prob{\abs{\mathfs{C}_0(t)}\le n+1} &= \sum_{k=0}^{n}{\Prob{\abs{\mathfs{C}_0(t)} = k+1}} = \sum_{k=0}^{n}\rb{(k+1)\rb{\Delta_{k}-\Delta_{k+1}}} \\
        &= \sum_{k=0}^{n}{(k+1)\Delta_k}-\sum_{k=0}^{n}{(k+1)\Delta_{k+1}} \\
        &= \sum_{k=0}^{n}{(k+1)\Delta_k}-\sum_{k=1}^{n+1}{k\Delta_{k}} \\
        &= -(n+1)\Delta_{n+1} + \Delta_0 + \sum_{k=1}^{n}\rb{(k+1-k)\Delta_k} \\
        &= -(n+1)\Delta_{n+1} + \Delta_0 + \sum_{k=1}^{n}\Delta_k \\
        &= -(n+1)\Delta_{n+1} + \sum_{k=0}^{n}\rb{\Prob{\con{0}{k}{t}} - \Prob{\con{0}{k+1}{t}}} \\
        &= -(n+1)\Delta_{n+1} + \Prob{\con{0}{0}{t}} - \Prob{\con{0}{n+1}{t}} \\
        &= 1-(n+1)\Delta_{n+1} - \Prob{\con{0}{n+1}{t}}
    \end{aligned}
\end{equation}
To calculate this, we want to estimate $\Prob{\con{0}{n}{t}}$.
Notice that until the clusters of 0 and n connect, $X_0$ and $X_n$ are independent random walks with rate 1. Therefore the difference between the two is a random walk with rate 2. 

$$\Prob{0 \overset{t}{\leftrightarrow} n+1} = \Prob{X_{n+1}(t)-X_0(t)=n} = \Prob{\min_{0\le j\le t}\{S_j\}\le n}$$

Where $S_t$ is a symmetric rate 2 random walk starting at $S_0 \sim X_{n+1}-X_{0} \sim \sum_{i=1}^{n}{G_i}$ where $G_i \sim Geo(p)$, with $\Prob[S]{\cdot}$ being its probability measure.  
Denote $\frac{1}{p} = \eta$, and then for $\epsilon = n^{-\frac{1}{4}}$:


\begin{equation}
    \begin{aligned}
        &\Prob[S]{\min_{0\le j\le t}\{S_j\}\le n} =\\
        &\Prob[S]{\min_{0\le j\le t}\{S_j\}\le n|S_0 - n \in [n(\eta-1-\epsilon),n(\eta-1+\epsilon)] }\cdot\\
        &\Prob[S]{S_0-n\in [n(\eta-1-\epsilon),n(\eta-1+\epsilon)]} + O(e^{-n^{\frac{1}{2}}})  = \\
        &\Prob[S]{\min_{0\le j\le t}\{S_j\}\le 0|S_0 \in [n(\eta-1-\epsilon),n(\eta-1+\epsilon)]}\cdot\\
        &\Prob[S]{S_0 \in [n(\eta-\epsilon),n(\eta+\epsilon)]} + O(e^{-n^\frac{1}{2}})
    \end{aligned}
\end{equation}

Since by hoeffding
\begin{equation}\label{hoeffdingBinom}
    \begin{aligned}
        \Prob{\sumn{G_i} > n\eta(1+\ep)} 
        &= \Prob{Bin(n(\eta+\ep),p) < n} \\
        &= \Prob{Bin(n(\eta+\ep),p)-n(1+p\ep) < np\ep} \\ 
        &\le e^{\frac{2n^2p^2\ep^2}{(1+\ep)n}} \le e^{-n\ep^2p^2} = O(e^{-n^{\frac{1}{2}}}) 
    \end{aligned}
\end{equation}
And similarly for the other direction.

From this, using LCLT (\cite{lawler2010random} Theorem 2.1.1) we get (denoting by $B(t)$ standard Brownian motion with law $\Prob[B]{\cdot}$):
\begin{equation}\label{LCLT}
    \begin{aligned}
        &\Prob[S]{\min_{0\le j\le t}\{S_j\}\le 0|S_0 \in [n(\eta-1-\epsilon),n(\eta-1+\epsilon)] } \\ 
        \le \ &\Prob[S]{\max_{0\le j\le 1}\cb{\frac{S_{tj}}{2\sqrt{t}}} \ge \frac{n(\eta-1+\epsilon)}{2\sqrt{t}}} \\
        = \ &2\Prob[S]{\frac{S_{tj}}{2\sqrt{t}} \ge \frac{n(\eta-1+\epsilon)}{2\sqrt{t}}} \\
        = \ &2\Prob[B]{B(s) \ge \frac{n(\eta-1+\epsilon)}{2\sqrt{t}}} + O(t^{-3/2}) \\
        = \ &2\int_{\frac{n(\eta-1)}{2\sqrt{t}}+\frac{n}{2\sqrt{t}}\epsilon}^{\infty}{\frac{1}{\sqrt{2\pi}} e^{-\frac{x^2}{2}}dx} + O(t^{-3/2})
    \end{aligned}
\end{equation}

Now we can use this approximation:

\begin{equation}
    \begin{aligned}
        &(n+1)\Delta_{n+1} = (n+1)\rb{\Prob{\con{0}{n+1}{t}} - \Prob{\con{0}{n+2}{t}}}\\
        \le \ 
        &\frac{
            2\int_{\frac{n(\eta-1)}{2\sqrt{t}} + \frac{n}{2\sqrt{t}}\epsilon}^\infty 
            \frac{1}{\sqrt{2\pi}} e^{-\frac{x^2}{2}} \, dx 
            -
            2\int_{\frac{(n+1)(\eta-1)}{2\sqrt{t}} + \frac{n+1}{2\sqrt{t}}\epsilon}^\infty 
            \frac{1}{\sqrt{2\pi}} e^{-\frac{x^2}{2}} \, dx
            + O(t^{-3/2})
        }{\frac{1}{n+1}}
        + \frac{O(e^{-n^{\frac{1}{2}}}) }{n} \\
        = \ 
        &\frac{
            2\frac{c(\eta-1)}{2}\int_{\frac{n(\eta-1)}{2\sqrt{t}} + \frac{n}{2\sqrt{t}}\epsilon}^{\frac{(n+1)(\eta-1)}{2\sqrt{t}} + \frac{n+1}{2\sqrt{t}}\epsilon}
            \frac{1}{\sqrt{2\pi}} e^{-\frac{x^2}{2}} \, dx
        }{\frac{c(\eta-1)}{2}\frac{1}{n+1}}
        + O(t^{-1}) + \frac{O(e^{-n^{\frac{1}{2}}})}{n}  \\
        \xrightarrow{t \to \infty}
        &2\frac{c(\eta-1)}{2\sqrt{2\pi}}e^{-\frac{\big(\frac{c}{2}(\eta-1)\big)^2}{2}}
    \end{aligned}
\end{equation}

To compute the limit, we can see that the difference between the upper and lower boundary is:
$$
\frac{(n+1)(\eta-1)}{2\sqrt{t}}+\frac{n+1}{2\sqrt{t}}\epsilon
-
\frac{n(\eta-1)}{2\sqrt{t}}+\frac{n}{2\sqrt{t}}\epsilon
=
\frac{\eta-1}{2\sqrt{t}}+\frac{\epsilon}{2\sqrt{t}} \sim \frac{c(\eta-1)}{2}\frac{1}{n+1}
$$

We get the same limit with $\ge$ in the same way, and then by sandwich:
$$
(n+1)\Delta_{n+1} \xrightarrow{t \to \infty} \frac{c(\eta-1)}{\sqrt{2\pi}}e^{-\frac{\big(\frac{c}{2}(\eta-1)\big)^2}{2}} 
.$$
Estimate \ref{LCLT} also gives:
\begin{equation}
    \begin{aligned}
        \Prob{\con{0}{n+1}{t}}  
        =  &2\int_{\frac{n(\eta-1)}{2\sqrt{t}}+\frac{n}{2\sqrt{t}}\epsilon}^{\infty}{\frac{1}{\sqrt{2\pi}} e^{-\frac{x^2}{2}}dx} + O(t^{-3/2}) \\
        \xrightarrow{t \to \infty} &\int_{\frac{c(\eta-1)}{2}}^{\infty}{\frac{2}{\sqrt{2\pi}} e^{-\frac{x^2}{2}}dx} \\
        = &\int_{\frac{c(\eta-1)}{2}}^{\infty}{\frac{2}{\sqrt{2\pi}} e^{-\frac{x^2}{2}}dx} \\
        = &2-2\Phi\rb{\frac{c(\eta-1)}{2}}
    \end{aligned}
\end{equation}


Combining the results from above we get:
\begin{equation}
    \begin{aligned}
        \Prob{\frac{\abs{C_0(t)}}{\sqrt{t}} \le x} \xrightarrow{t \to \infty} &1 - \frac{x(\eta-1)}{\sqrt{2\pi}}e^{-\frac{\big(\frac{x}{2}(\eta-1)\big)^2}{2}} -2 +2\Phi\rb{\frac{x(\eta-1)}{2}} \\
        &= 2\Phi\rb{\frac{x(\eta-1)}{2}} - \frac{x(\eta-1)}{\sqrt{2\pi}}e^{-\frac{\big(\frac{x}{2}(\eta-1)\big)^2}{2}} - 1
    \end{aligned}
\end{equation}


\end{proof}

\section{$\alpha \ge 0$}
\subsection{upper bound}
Consider the setting of $d=1, \alpha \ge 0, p \in (0,1)$.
\begin{theorem}
Let $d=1,\alpha\ge 0,p\in(0,1)$. $\exists T>0$ s.t for all $\ep>0$ there exists $c>0$ s.t for all $t>T$,
\begin{equation}
    \Prob{|\mathfs{C}_0(t)| \le ct^{\frac{1}{\alpha + 2}}} \ge 1-\ep
.\end{equation}
\end{theorem}

The idea of the proof is to consider only slowing down the cluster after time intervals of $t^2$ and $t^{2+\alpha}$ time, as this gives us an upper bound on the size. By random walk properties, we get that $|\mathfs{C}_0(t^2)|\approx t$. After applying the slowdown to the clusters, we can see that it takes roughly $t^{2+\alpha}$ to perform the same amount of steps as in the initial $t^2$ time. This shows us that we still have $|\mathfs{C}_0(t^{2+\alpha})| \approx t$. We do this showing that after clusters $\mathfs{C}_0$ $\mathfs{C}_t$ connect, and $\mathfs{C}_{Ct-t}$ $\mathfs{C}_{Ct}$ connect, the two clusters $\mathfs{C}_0$ and $\mathfs{C}_{Ct}$ still have $\approx t^2$ steps to perform until they connect. See figure \ref{fig:alpha positive upper bound} for the relevant clusters.

\begin{figure}
    \begin{center}
        \begin{tikzpicture}
            \draw[thick] (-1,0) -- (6,0);
            
            \fill (0,0) circle (2pt) node[below] {$0$};
            \fill (1.5,0) circle (2pt) node[below] {$t$};
            \fill (4,0) circle (2pt) node[below] {$Ct-t$};
            \fill (5.5,0) circle (2pt) node[below] {$Ct$};
        \end{tikzpicture}
    \end{center}
    \caption{The clusters involved in the proof of the upper bound for $\alpha \ge 0$}
    \label{fig:alpha positive upper bound}
\end{figure}

\begin{proof}
Let $\ep > 0$. Using the approximations seen in the previous result \eqref{LCLT}, we can see that $\exists c_1,c_2,T>0$ s.t $\forall C>0,t>T$:
\begin{equation*}
        A_1 = \cb{S_0^t + S_t^0 \le c_1t^2-1}, A_2 = \cb{S_{Ct-t}^{Ct} + S_{Ct}^{Ct-t} \le c_1t^2-1}, A_3 = \cb{S_{0}^{Ct} +S_{Ct}^0 \ge 2c_2(Ct)^2}\\
\end{equation*}
for i=1,2,3:
\begin{equation*}
    \Prob{A_i} \ge 1-\ep
.\end{equation*}
This is because it holds that:
\begin{equation}\label{boundEventsA}
    \begin{aligned}
        \Prob{A_1} = \Prob{A_2} \ge 1-c_ge^{-c_gc_1t}-c_ge^{-c_gc_1} \ge 1-c_ge^{-c_gc_1} \ge 1-\ep\\
        \Prob{A_3} \ge 1-c_ge^{-c_gc_2Ct}-c_ge^{-c_g\frac{1}{c_2}} \ge 1-c_ge^{\frac{1}{c_2}} \ge 1-\ep
    \end{aligned}
\end{equation}
Using Hoeffding calculation \eqref{hoeffdingBinom} with $\ep = \frac{\eta-1}{2}$ and \cite[Proposition 2.4.5.]{lawler2010random}, where $c_g$ is some global constant (not always the same constant). For the second inequality to hold, we take $C>\frac{1}{(c_2)^2}$ and $T>1$.

Define the events:
\begin{equation*}
    A_4 = \cb{S_0^{Ct} \ge c_2(Ct)^2}, A_5 = \cb{S_{Ct}^0 \ge c_2(Ct)^2}, T_1 = \cb{\tau_0^{Ct}\le t^\gamma}, T_2 = \cb{\tau_{Ct}^0 \le t^\gamma}
.\end{equation*}
And notice that $A_3 \subseteq A_4 \cup A_5$, since one of the two clusters must have walked at least half of the steps. Now we can see that for $t>T$:

\begin{equation}\label{eq:connection_prob}
    \begin{aligned}
        \Prob{0 \overset{t^\gamma}{\leftrightarrow} Ct} 
        &= \Prob{\tau_0^{Ct} \le t^\gamma, \tau_{Ct}^0 \le t^\gamma} \\
        &= \Prob{T_1 \cap T_2} \\
        &\le \Prob{T_1 \cap T_2 \cap (A_4 \cup A_5)} + \ep \\
        &= \Prob{(T_1 \cap T_2 \cap A_4) \cup (T_1 \cap T_2 \cap A_5)} + \ep \\
        &\le \Prob{T_1 \cap T_2 \cap A_4} + \Prob{T_1 \cap T_2 \cap A_5} + \ep \\
        &= 2\Prob{T_1 \cap T_2 \cap A_4} + \ep \\
        &\le 2\Prob{T_1 \cap A_4} + \ep \\
        &\le 2\Prob{T_1 \cap A_4 \cap A_1} + 3\ep \\
    \end{aligned}
\end{equation}

Using the law of total probability. Next we upper bound the last probability in \eqref{eq:connection_prob}.

\begin{equation}
    \begin{aligned}
        \Prob{T_1 \cap A_4 \cap A_1}
        &= \Prob{\sum_{i=1}^{S_0^{Ct}} t_i^0 \le t^\gamma, S_0^{Ct} \ge c_2(Ct)^2, S_0^{t} + S_t^{0} \le c_1t^2-1} \\
        &\le \Prob{\sum_{i=c_1t^2+1}^{c_2(Ct)^2} t_i^0 \le t^\gamma, S_0^{Ct} \ge c_2(Ct)^2, S_0^{t} + S_t^{0} \le c_1t^2-1} \\
        &\le \Prob{\sum_{i=c_1t^2+1}^{c_2(Ct)^2} t_i^0 \le t^\gamma, S_0^{t} + S_t^{0} \le c_1t^2-1} \\
        &\le \Prob{T_{c_1t^2,c_2(Ct)^2}^0 \le t^\gamma, \abs{C_0\rb{T_{c_1t^2}^0}} \ge t} \\
    \end{aligned}
,\end{equation}
where the last inequality is due to the fact that regardless which cluster has made the last move before connecting, connection will occur by time $T_{c_1t^2}^0=\sum_{i=1}^{c_1t^2}{t_i^0}$ (see Lemma \ref{lem:sizeBoundTime}).

Now we use \eqref{slow>0} to get that:

\begin{equation}
    \begin{aligned}
        &\le \Prob{t^\alpha \Tilde{T}_{c_1t^2,c_2(Ct)^2}^0\le t^\gamma} \\
        &=   \Prob{\sum_{i=c_1t^2+1}^{c_2(Ct)^2} \Tilde{t}_i^0 \le t^{\gamma-\alpha}} \\
        &\le \ep
    \end{aligned}
,\end{equation}

since this is a sum of $n = c_2(Ct)^2-c_1t^2$ i.i.d R.Vs, we can choose $C$ large enough so that $n \ge 2t^2$. Then by WLLN we have that $\frac{t^{\gamma-\alpha}}{n} \le \frac{t^2}{2t^2} = \frac{1}{2}$ and thus we get convergence to 0. We can also see that $n \xrightarrow[C \to \infty]{} \infty$, so we can take $C$ large enough for the last inequality to hold.


\end{proof}

\subsection{lower bound}
\begin{theorem}
    \begin{equation}
        \lim_{c \to 0}\lim_{t \to \infty}{\Prob{|\mathfs{C}_0(t)|\ge ct^{\frac{1}{\alpha + 2}}}} = 1
    \end{equation}
\end{theorem}

The general idea for the proof is to couple to the cluster cluster model with $\alpha=0$. Then, until $\mathfs{C}_0$ reaches size $ct^{\frac{1}{\alpha+2}}$, we slow it to the rate of a cluster of that size. We can see that this slows down the cluster, and using the coupling we show that this gives a lower bound on the size.

\begin{proof}
Let $1>c>0$, $t\ge 0$. Denote by $\cb{\mathfs{A}_{s}^0}_{s\ge 0}$ a Cluster Cluster process with parameter $\alpha$. For the lower bound, we can consider the coupling between $\cb{\mathfs{A}_{s}^0}_{s\ge 0}$ and $\cb{\mathfs{A}_{s(2t)^\alpha}^\alpha}_{s\ge 0}$. 

For both processes, consider $S_u = X_{ct}(u)-X_0(u)$, denoting them by $S_u^{\alpha > 0}$ and $S_u^{\alpha = 0}$ accordingly. Note that $S_u$ is a SRW, absorbed at $ct$, with step times according to a non homogeneous Poisson process with rate $\lambda(u) = \abs{\mathfs{C}_0(u)}^{-\alpha}+\abs{\mathfs{C}_{ct}(u)}^{-\alpha}$. We couple the two (and thus couple the two processes) in the following way.

First we create a Discrete random walk $W_u$, and a rate 1 Poisson process $N(t)$ independently of each other. Then we create $S_u$ by taking the inter arrival times of $N(u(t))$, denoted by $t_i$, as the step times (again here $u(t) = \int_{0}^{t}{\lambda(s) ds}$). Then we decide the step direction according to $W_i$.

We first show that under the above coupling, if for $\mathfs{A}^0$ we have $0 \overset{t}{\leftrightarrow} 2ct$, then for $\mathfs{A}^\alpha$ we have $\abs{\mathfs{C}_0(t)} \ge ct$. We assume by contradiction that at time $t$ we have $0 \overset{t}{\leftrightarrow} 2ct$ and $\abs{\mathfs{C}_0(t)} < ct$. This means that that in the $\mathfs{A}^\alpha$, up to time $t$ the rate of $S_u$ was at least $\rb{ct}^{-\alpha} \ge t^{-\alpha}$, so $u^{\alpha > 0}(\cdot) \ge u^{\alpha = 0}(\cdot)$ which gives that $t_i^{\alpha > 0} \le t_i^{\alpha = 0}$. By the coupling in both processes it takes the same amount of steps, given by $W_u$, until the two clusters connect. Since all the step times in the case $\alpha \ge 0$ are faster up to time t, we get that $0 \overset{t}{\leftrightarrow} 2ct$ in the dynamic process which implies $\abs{\mathfs{C}_0(t)} \ge 2ct$ which is a contradiction.


From this we get that (after adjusting for time scaling):
\begin{equation}
    \Prob[\alpha > 0]{|\mathfs{C}_0(2t^{2+\alpha})| \ge ct} \ge \Prob[\alpha = 0]{0 \overset{t^2}{\leftrightarrow} 2ct}
\end{equation}
We can combine this with the fact that (by symmetry and inclusion)
\begin{equation}
    \Prob[\alpha = 0]{0 \overset{t^2}{\leftrightarrow} 2ct} \ge \frac{1}{2} \Prob{\abs{\mathfs{C}_0(t^2)} \ge 4ct}
\end{equation}
and with Theorem \ref{alpha=0} to get that for all $\epsilon > 0$, exists $1>c>0$ s.t:
$$\lim_{t \to \infty}{\Prob[\alpha > 0]{|\mathfs{C}_0(t)|\ge ct^{\frac{1}{\alpha + 2}}}} \ge 1-\epsilon$$
\end{proof}
\section{$\alpha \le 0$}
\subsection{lower bound}
Define recursively $\gamma_1 = \oov{2}$, $\gamma_{n+1} = \frac{1-\alpha\gamma_n}{2}$.

\begin{theorem}
    For $-2 < \alpha < 0$, $\forall \ep>0, n\in \BN, \exists c>0$ s.t:
    \begin{equation}
        \lim_{t \to \infty}{\Prob{\abs{\mathfs{C}_0(t)} \le ct^{\gamma_n}}} \le \ep  
    \end{equation}
    For $\alpha < -2$:
    \begin{equation}
        \lim_{t \to \infty}{\Prob{\abs{\mathfs{C}_0(t)} < \infty}} = 0
    \end{equation}
\end{theorem}

The general idea for the proof, is that we consider the difference random walk of cluster $\mathfs{C}_0$ with all clusters $\mathfs{C}_{C_nt^{\gamma_n}}$ for correctly chosen constants $C_n$, and $\gamma_n$ as described above. 
We ensure the amount of steps is correct in the same way as above (with error rate $\frac{\ep}{n^2}$ which converges), and then use the speed up of earlier connections when bounding the time the later connections take. 
We can see that $\gamma_n$ is chosen exactly so that after we apply the speed up (i.e Lemma \ref{lem:M bounded size alpha<0}) we get exactly the correct scale. 
See figure \ref{fig:alpha negative lower bound} for the relevant clusters.

\begin{figure}
    \begin{center}
        \begin{tikzpicture}
            \draw[thick] (-1,0) -- (8,0);
            
            \fill (0,0) circle (2pt) node[below] {$0$};
            \fill (1,0) circle (2pt) node[below] {$C_1 t^{\gamma_1}$};
            \fill (2.2,0) circle (2pt) node[below] {$C_2 t^{\gamma_2}$};
            \fill (4,0) circle (2pt) node[below] {$C_3 t^{\gamma_3}$};
            
            \node at (5.5,-0.3) {$\dots$};

            \fill (7,0) circle (2pt) node[below] {$C_n t^{\gamma_n}$};
        \end{tikzpicture}
    \end{center}
    \caption{The clusters involved in the proof of the lower bound for $\alpha \le 0$}
    \label{fig:alpha negative lower bound}
\end{figure}

\begin{proof}
Let $\ep \ge 0$, $n \in \BN$ and $\epsilon_n = \frac{\ep}{n^2}$. Let $C_n>0$ (to be chosen later). Let $a_n \to \infty$ not a function of $C_n$ such that: 
$$A_n = \cb{S_0^{C_nt^{\gamma_n}} + S_{C_nt^{\gamma_n}}^0 \le a_n\rb{C_nt^{\gamma_n}}^2}$$
$$\Prob{A_n} \ge 1- \epsilon_n,$$
where we ensure the existence of $a_n$ by \eqref{boundEventsA}, requiring that
\begin{equation}\label{NegativeAlphaCond1}
    \begin{aligned}
        1-c_ge^{-c_ga_n} &\ge 1 - \frac{\ep}{n^2}\\
        a_n &\ge c_g\ln{\frac{n^2}{c_g\ep}},
    \end{aligned}
\end{equation}
and also requiring 
\begin{equation}\label{NegativeAlphaCond2}
    C_nt^{\gamma_n} > 1.
\end{equation}


Using this, we can see that for all $t>0$ large enough, for any constant $c_g>0$, by law of total probability on the event $A_n$:

\begin{equation}\label{eq:boundiidlength}
    \Prob{\sum_{i=S_0^{C_nt^{\gamma_n}} +1}^{S_0^{C_{n+1}t^{\gamma_{n+1}}}}{t_i^0}>\frac{t}{c_g(n+1)^2}} 
    \le
    \Prob{\sum_{i=S_0^{C_nt^{\gamma_n}} + 1}^{a_nC_{n+1}^2t^{2\gamma_{n+1}}}{t_i^0}>\frac{t}{c_g(n+1)^2}} + \epsilon_n
\end{equation}

By Lemma \ref{lem:sizeBoundTime}:

\begin{equation}\label{eq:boundiidlength2}
    \begin{aligned}
        &\le \Prob{\oov{N}\sum_{i=1}^{N}{\Tilde{t}^0_i} > C_n^{-\alpha}t^{-\gamma_n\alpha}t\oov{c_g (n+1)^2 a_n C_{n+1}^2t^{2\gamma_{n+1}}}}+ \epsilon_n\\
        &= \Prob{\oov{N}\sum_{i=1}^{N}{\Tilde{t}^0_i} > \frac{C_n^{-\alpha}}{c_g (n+1)^2 a_n\rb{C_{n+1}}^2}}+ \epsilon_n,
    \end{aligned}
\end{equation}

where $N = a_nC_{n+1}^2t^{2\gamma_{n+1}}$ the number of summed R.Vs, which bounds from above the number of elements in the RHS of \eqref{eq:boundiidlength}.
We now require
\begin{equation}\label{NegativeAlphaCond3}
    \frac{C_{n}^{-\alpha}}{c_g (n+1)^2 a_n(C_{n+1})^2} \ge 2.
\end{equation}
Since $exp(1)$ R.V is sub-exponential, using Bernstein's inequality

\begin{equation}
    \begin{aligned}
        \Prob{\oov{N}\sum_{i=1}^{N}{\Tilde{t}^0_i} > \frac{C_n^{-\alpha}}{c_g (n+1)^2 a_n\rb{C_{n+1}}^2}}
        \le 
        \Prob{\oov{N}\sum_{i=1}^{N}{\Tilde{t}^0_i}-1 > 1}
        \le
        e^{-c_gN}
        \le
        \ep_n
    \end{aligned}
\end{equation}

where to satisfy the last inequality we add the condition

\begin{equation}
    \begin{aligned}\label{NegativeAlphaCond4}
        e^{-c_gN} &\le \frac{\ep}{n^2} \\
        N &\ge c_g\ln{\frac{n^2}{\ep}} \\
        a_nC_{n+1}^2t^{2\gamma_{n+1}} &\ge c_g\ln{\frac{n^2}{\ep}}.
    \end{aligned}
\end{equation}

Overall we get

\begin{equation}\label{eq:eqbound}
    \begin{aligned}
        \Prob{\sum_{i=S_0^{C_nt^{\gamma_n}}}^{S_0^{C_{n+1}t^{\gamma_{n+1}}}}{t_i^0}>t\frac{c_g}{(n+1)^2}} \le 2\epsilon_n.
    \end{aligned}
\end{equation}

We must satisfy conditions \eqref{NegativeAlphaCond1}, \eqref{NegativeAlphaCond2}, \eqref{NegativeAlphaCond3}, \eqref{NegativeAlphaCond4}. 
For condition \eqref{NegativeAlphaCond1} we take $a_n = g_\ep\ln{n}$ for constant $g_\ep$ large enough depending only on $\ep$. 
For condition \eqref{NegativeAlphaCond3} we take $C_n = g_c n^{-\frac{3}{\alpha+2}}$ for constant $g_c$ s.t if we plug into \eqref{NegativeAlphaCond3}:
\begin{equation}
    \begin{aligned}
        \frac{g_c n^{\frac{3\alpha}{\alpha+2}}}{c_g n^2 g_\ep\ln{n} g_c^2 (n+1)^{-\frac{6}{\alpha+2}}} \ge \frac{n^{\frac{3\alpha}{\alpha+2}-2+\frac{6}{\alpha+1}}}{c_gg_cg_\ep\ln{n}} = \frac{n}{c_gg_cg_\ep\ln{n}} \ge \frac{1}{c_gg_cg_\ep} \ge 2 .
    \end{aligned}
\end{equation}
Note that the justification for the first inequality differs between cases $\alpha<-2$ and $\alpha>-2$.

For the last two conditions \eqref{NegativeAlphaCond2}, \eqref{NegativeAlphaCond4}, we first consider the case $-2 < \alpha < 0$. 
In this case, for any fixed $n\in \BN$, we have $C_nt^{\gamma_n} \to \infty$ and $a_nC_{n+1}^2t^{2\gamma_{n+1}} \to \infty$ as $t \to \infty$. 
Given $K \in \BN$, we can have these conditions satisfied $\forall n\le K$ by taking $t$ large enough.  

For the case $\alpha < -2$, if $t>1$ then $C_nt^{\gamma_n} \to \infty$ and $a_nC_{n+1}^2t^{2\gamma_{n+1}} \to \infty$ as $n \to \infty$. 
This is because $\gamma_n = \Omega(n)$ resulting in $t^{\gamma_n} = \Omega(e^n)$, and $a_n = \Omega(1)$, $C_n = \Omega(n^{-\frac{3}{\alpha+2}})$. 
As a result, for some fixed $K \in \BN$, we can have these condition satisfied $\forall n>K$ and $t>1$. For $n<K$ we can take $t$ large enough in the same way as in the case of $-2 < \alpha$.   

To summarize, if $\alpha < -2$, then $\exists T>0$ s.t for $t>T$, \eqref{eq:eqbound} holds $\forall n \in \BN$, and if $0 > \alpha > -2$ then $\forall N, \exists T>0$, s.t for $t>T$, \eqref{eq:eqbound} holds $\forall n<N$.

Now we prove the main claim:

\begin{equation}
    \Prob{0 \overset{t}{\nleftrightarrow} C_nt^{\gamma_n}} 
    = 
    \Prob{\sum_{i=1}^{S_0^{C_nt^{\gamma_n}}}{t_i^0}>t\, \, \text{or} \sum_{i=1}^{S_{C_nt^{\gamma_n}}^0}{t_i^{C_nt^{\gamma_n}}}>t} \le 2\Prob{\sum_{i=1}^{S_0^{C_nt^{\gamma_n}}}{t_i^0}>t},
\end{equation}
where in the first equality follows from the definitions of $S_0^x$ - a connection occurs if and only if both clusters have completed the steps required for connection before time $t$.

We can split this to $n$ different sums (for simplicity $\gamma_0 = 0, C_0 = 0, S^0_0 = 0$), and use \eqref{eq:eqbound} with constant $c^* = \rb{\sum_{i=1}^{\infty}{i^{-2}}}^{-1}$: 
\begin{equation}
    \begin{aligned}
        \Prob{\sum_{i=1}^{S_0^{C_nt^{\gamma_n}}}{t_i^0}>t}
        &\le
        \Prob{\bigcup_{j=1}^{n}\cb{\sum_{i={S_0^{C_{j-1}t^{\gamma_{j-1}}}}+1}^{S_0^{C_jt^{\gamma_j}}}{t_i^0}>t\frac{c^*}{j^2}}}\\
        &\le
        \sum_{j=1}^{n}{\Prob{\sum_{i={S_0^{C_{j-1}t^{\gamma_{j-1}}}}+1}^{S_0^{C_jt^{\gamma_j}}}{t_i^0}>t\frac{c^*}{j^2}}} \\
        &\le
        2\sumn{\epsilon_i}\\
        &\le 
        c_g\ep.
    \end{aligned}
\end{equation}
For the first inequality, we used the fact that for any $n$ numbers $x_i \in \mathds{R}, i \in [n]$, and any $p \in \Delta_n$, for any value $v\in \mathds{R}$ it holds that
\begin{equation}
    \sum_{i=1}^{n} x_i > v \Rightarrow \exists i\in [n]: x_i > p_i v.
\end{equation}
Otherwise we would get that $\forall i\in [n]: x_i \le p_i v$, which gives $\sumn{x_i} \le v$.

Using this bound, we can see that
\begin{equation}
    \Prob{|\mathfs{C}_0(t)|\le C_nt^{\gamma_n}} \le \Prob{0 \nleftrightarrow C_nt^{\gamma_n}, 0\nleftrightarrow - C_nt^{\gamma_n}} \le 2\Prob{0 \nleftrightarrow C_nt^{\gamma_n}}
    \le c_g\ep.
\end{equation}
In the case $-2<\alpha<0$, this will give us that for any fixed $n$, $\exists T$ s.t $\forall t>T$:
\begin{equation}
    \Prob{|\mathfs{C}_0(t)|\le C_nt^{\gamma_n}} \le \epsilon.
\end{equation}
In the case $\alpha < -2$, since $\exists T$ s.t for $t>T$ this holds $\forall n$, we can take $n \to \infty$ and get that 
\begin{equation}
    \Prob{|\mathfs{C}_0(t)| < \infty} \le \epsilon.
\end{equation}
By the definition of limit, this exactly gives us that
\begin{equation}
    \lim_{t \to \infty}{\Prob{|\mathfs{C}_0(t)| < \infty}} = 0. 
\end{equation}
\end{proof}

\subsection{upper bound}
\begin{lemma}\label{lem:M bounded size alpha<0}
    Let $M>0, -2<\alpha < 0$ and consider the process $\mathfs{A}_t^M$ having rates upper bounded by $M$.
    $\forall \ep>0$, exists $\exists c>0,T>0$ such that $\forall t>T $
    \begin{equation}
        \Prob{\abs{\mathfs{C}_0(t)} \ge ct^{\frac{1}{\alpha+2}}} \le \ep ~ ,
    \end{equation}
    where $c = c_\alpha\log{\ep^{-1}}$, and $c_\alpha$ is a constant that may depend on $\alpha$.
\end{lemma}

The general idea of the proof is to consider the clusters $\mathfs{C}_{j_n}$, where the $j_n$ are chosen with increased spacing so that we can union bound the noise resulting from the random walk between each pair of Clusters $\mathfs{C}_{j_n}, \mathfs{C}_{j_{n+1}}$. 
See figure \ref{fig:negative alpha upper bound} for the relevant clusters.
Then we show that none of the clusters $\mathfs{C}_{j_n}$ have enough time to complete the steps needed. 

The way to show this is by first consider the event that $\mathfs{C}_{j_0}$ connects to $\mathfs{C}_{j_1}$. 
Then we split into cases - if $\mathfs{C}_{j_0}$ walked more steps towards $\mathfs{C}_{j_1}$, then we know those steps were taking before the connection, hence we have a bound on the size at the time of the steps and can use lemma \ref{lem:sizeBoundTime}. 
If $\mathfs{C}_{j_1}$ took more steps, then either it connected to $\mathfs{C}_{j_2}$ in the relevant time or not. 
Again, if it didn't connect to $\mathfs{C}_{j_2}$, we have a bound on the size during the steps and can use lemma \ref{lem:sizeBoundTime}. 
Now if we require the connection between $\mathfs{C}_{j_1}$ and $\mathfs{C}_{j_2}$, then we are back at the point we started at except on step to the right. 
We can keep doing this inductively and be left with the event that $\mathfs{C}_{j_0}$ is connected to $\mathfs{C}_{j_n}$ which goes to $0$ in probability as $n \to \infty$.

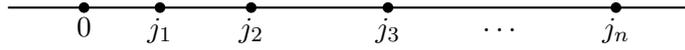
\begin{figure}
    \begin{center}
        \begin{tikzpicture}
            \draw[thick] (-1,0) -- (8,0);
            
            \fill (0,0) circle (2pt) node[below] {$0$};
            \fill (1,0) circle (2pt) node[below] {$j_1$};
            \fill (2.2,0) circle (2pt) node[below] {$j_2$};
            \fill (4,0) circle (2pt) node[below] {$j_3$};
            
            \node at (5.5,-0.3) {$\dots$};
            
            \fill (7,0) circle (2pt) node[below] {$j_n$};
        \end{tikzpicture}
    \end{center}
    \caption{The relevant clusters for the upper bound for $-2 < \alpha < 0$}
    \label{fig:negative alpha upper bound}
\end{figure}
        
\begin{proof}
    
Let $\ep > 0$. Define $\ep_i = \frac{\ep}{i^2}$ and consider the particles/clusters $\cb{\mathfs{C}_{j_i}}_{i \in \mathds{Z}}$ for some fixed $\cb{j_i}_{i \in \mathds{Z}}$. We denote $\Delta j_i = j_{i+1}-j_i$ and choose $j_i$ s.t $\Delta j_i = \Delta j_{-i}$ to ensure symmetry.
We now define the following events:
\begin{gather}\label{ABevents}
    A_i = \cb{S_{j_i}^{j_{i+1}}+S_{j_{i+1}}^{j_i} \ge 2\zeta_i\rb{\Delta j_i}^2} ,
    A = \bigcap_{i=-\infty}^{\infty}{A_i}                      \\
    B_i^r = \cb{S_{j_i}^{j_{i+1}} \ge S_{j_{i+1}}^{j_i}},
    B_i^\ell = \cb{S_{j_i}^{j_{i-1}} \ge S_{j_{i-1}}^{j_i}} ,
    B_i = B_i^\ell \cap B_i^r                                  \\
    G_i = \cb{j_i \overset{t^\gamma}{\leftrightarrow} j_{i+1}} = \cb{T^{j_i}_{S_{j_i}^{j_{i+1}}} < t^\gamma} \cap \cb{T^{j_{i+1}}_{S_{j_{i+1}}^{j_i}} < t^\gamma}
\end{gather}
And by \eqref{boundEventsA}, we choose $\zeta_i$ and $j_i$ s.t $\Prob{A_i} \ge 1-\ep_i$, and thus $\Prob{A} \ge 1-\ep$.

Now we begin the main argument of the proof. We will use the following bounds which will be proved later:
\begin{equation}\label{eq:boundsHelp}
    \begin{aligned}
        \Prob{A,B_{n+1}^\ell, G_n, G_{n+1}^C} &\le \ep_n \\
        \Prob{A,B_n,G_n} &\le \ep_{n} \\
        \Prob{A,G_0,G_{-1}^C,G_{1}^C} &\le \ep_0.
    \end{aligned}
\end{equation}
Using these bounds, we split the following event to several disjoint events using $B_i^\ell$ and $B_i^r$:
\begin{equation}\label{mainArgument}
    \begin{aligned}
        \Prob{G_0} 
        &\le \Prob{A,G_0} +\ep \\
        &\le \Prob{A,G_0, B_0} + \Prob{A,G_0,B_0^\ell,\rb{B_0^r}^C} + \Prob{A,G_0,\rb{B_0^\ell}^C,B_0^r}\\ &+ \Prob{A,G_0,\rb{B_0^\ell}^C,\rb{B_0^r}^C} + \ep \\
        &= \Prob{A,G_0, B_0} + 2\Prob{A,G_0,B_0^\ell,\rb{B_0^r}^C} + \Prob{A,G_0,\rb{B_0^\ell}^C,\rb{B_0^r}^C} + \ep,
    \end{aligned}
\end{equation}
where we used symmetry for the last equality. Now we further split using $G_1,G_{-1}$ and use symmetry:
\begin{equation}\label{mainArgument2}
    \begin{aligned}
        &\le \Prob{A,G_0, B_0} + 2\rb{\Prob{A,G_0,B_0^\ell,\rb{B_0^r}^C, G_1^C} + \Prob{A,G_0,B_0^\ell,\rb{B_0^r}^C, G_1}} \\ 
        &+ \Prob{A,G_0,\rb{B_0^\ell}^C,\rb{B_0^r}^C, G_1^C,G_{-1}} + \Prob{A,G_0,\rb{B_0^\ell}^C,\rb{B_0^r}^C, G_1,G_{-1}^C} \\
        &+ \Prob{A,G_0,\rb{B_0^\ell}^C,\rb{B_0^r}^C, G_1,G_{-1}} + \Prob{A,G_0,\rb{B_0^\ell}^C,\rb{B_0^r}^C, G_1^C,G_{-1}^C} + \ep\\
        &\le \ep_0 + 2\rb{\ep_0 + \Prob{A,G_0,\rb{B_0^r}^C, G_1}} \\ 
        &+ \Prob{A,G_0,\rb{B_0^\ell}^C,G_{-1}} + \Prob{A,G_0,\rb{B_0^r}^C, G_1} \\
        &+ \Prob{A,G_0,\rb{B_0^r}^C, G_1} + \ep_0 \\
        &= 4 \ep_0 + 5\Prob{A,G_0,\rb{B_0^r}^C, G_1} + \ep \\
        &\le 10\ep_0 + 5\Prob{A,G_0,B_1^\ell, G_1} + \ep,
    \end{aligned}
\end{equation}

using the bounds in \eqref{eq:boundsHelp}.

From here we can show the following by induction. For any $n\in \BN$:

\begin{equation}\label{inductionMain}
    \begin{aligned}
        \Prob{G_0} \le 10\sum_{i=0}^{n-1}{\ep_i} + 5\Prob{A,\bigcap_{i=0}^{n-1}{G_i},B_n^\ell, G_n} + \ep
    \end{aligned}
\end{equation}
The base case $n=1$ is shown above. As for the induction step, consider the following:
\begin{equation}\label{inductionStep}
    \begin{aligned}
        & \quad \ \Prob{A,\bigcap_{i=0}^{n-1}{G_i},B_n^\ell, G_n} \\
        &\le \Prob{A,\bigcap_{i=0}^{n-1}{G_i},B_n^\ell, B_n^r, G_n} + \Prob{A,\bigcap_{i=0}^{n-1}{G_i},B_n^\ell, \rb{B_n^r}^C, G_n} \\
        &\le  \ep_n + \Prob{A,\bigcap_{i=0}^{n-1}{G_i},B_n^\ell, \rb{B_n^r}^C, G_n} \\
        &\le  \ep_n + \Prob{A,\bigcap_{i=0}^{n-1}{G_i},B_n^\ell, \rb{B_n^r}^C, G_n, G_{n+1}^C} \\
        & \quad \quad \ + \Prob{A,\bigcap_{i=0}^{n-1}{G_i},B_n^\ell, \rb{B_n^r}^C, G_n, G_{n+1}} \\
        &\le 2\ep_{n} + \Prob{A,\bigcap_{i=0}^{n}{G_i},B_{n+1}^\ell, G_{n+1}} 
    \end{aligned}
\end{equation}
Applying this gives us the induction step.

Now we will show using Lemma \ref{lem:sizeBoundTime} that the bounds in \eqref{eq:boundsHelp} hold. We can see that:
\begin{equation}\label{eventTypes1}
    \begin{aligned}
        G_n \cap A \cap B_{n+1}^\ell \cap G_{n+1}^C 
        \subseteq \cb{T^{j_{n+1}}_{S_{j_{n+1}}^{j_n}} < t^\gamma,
        S_{j_{n+1}}^{j_n} \ge \zeta_n\rb{\Delta j_n}^2, \abs{\mathfs{C}_{j_{n+1}}\rb{T^{j_{n+1}}_{S_{j_{n+1}}^{j_n}}-}}\le \Delta j_n + \Delta j_{n+1}},
    \end{aligned}
\end{equation}

since if $\mathfs{C}_{j_{n+2}}$ and $\mathfs{C}_{j_{n+1}}$ don't connect, then $\mathfs{C}_{j_{n+1}}$ is at most of size $\Delta j_n + \Delta j_{n+1}$ until $\mathfs{C}_{j_n}$ and $\mathfs{C}_{j_{n+1}}$ connect. The combination of $A$ and $B_{n+1}^\ell$ gives a lower bound on the number of steps by $\mathfs{C}_{j_{n+1}}$ until connecting to $\mathfs{C}_{j_n}$. This can be seen in figure \ref{fig:cases}.

We can also see that:

\begin{equation}\label{eventTypes2}
    \begin{aligned}
        G_n \cap A \cap B_n 
        &\subseteq 
        \cb{T^{j_n}_{S_{j_n}^{j_{n+1}}}< t^\gamma,
        S_{j_n}^{j_{n+1}} \ge S_{j_n}^{j_{n-1}} \ge c_2^{n-1}\rb{\Delta j_{n-1}}^2, \abs{\mathfs{C}_{j_n}\rb{T_{c_2^n(\Delta j_{n-1})^2}^{j_n}-}}\le \Delta j_{n-1} + \Delta j_{n}},
    \end{aligned}
\end{equation}
by definition of $B_n$ combined with $A$. See figure \ref{fig:cases}.
The last observation is that:

\begin{equation}\label{eventTypes3}
    \begin{aligned}
        G_0 \cap A \cap G_{-1}^C \cap G_1^C 
        &\subseteq 
        \cb{T^{j_0}_{S^{j_1}_{j_0}}< t^\gamma,
        S_{j_0}^{j_1} \ge c_2^0\rb{\Delta j_0}^2, \abs{\mathfs{C}_{j_0}\rb{T^{j_0}_{S^{j_1}_{j_0}}-}}\le \Delta j_0 + \Delta j_1} \\
        &\cup
        \cb{T^{j_1}_{S^{j_0}_{j_1}} < t^\gamma,
        S_{j_1}^{j_0} \ge c_2^0\rb{\Delta j_0}^2, \abs{\mathfs{C}_{j_1}\rb{T^{j_1}_{S^{j_0}_{j_1}}-}}\le \Delta j_0 + \Delta j_{-1}}
    \end{aligned}
\end{equation}

Since $G_{-1}^C$ and $G_1^C$ give us a bound on the size of $C_{j_0}$ and $C_{j_1}$ in the same way as before, and the rest is due to $A$ combined with either $B_{j_0}^r$ or $B_{j_1}^\ell$ (Since one of the two always happens).

\begin{figure}[h]
    \centering
    \begin{tikzpicture}[scale=1]
        \draw[thick] (0,0) -- (6,0);
        \fill (1,0) circle (2pt) node[below] {$j_n$};
        \fill (3,0) circle (2pt) node[below] {$j_{n+1}$};
        \fill (5,0) circle (2pt) node[below] {$j_{n+2}$};
        \draw[->, thick] (2.8,0.25) -- (1.2,0.25);
        \draw[thick] (4,-0.3) -- (4.4,0.3);
    \end{tikzpicture}
    \hspace{1cm} 
    \begin{tikzpicture}[scale=1]
        \draw[thick] (0,0) -- (6,0);
        \fill (1,0) circle (2pt) node[below] {$j_{n-1}$};
        \fill (3,0) circle (2pt) node[below] {$j_n$};
        \fill (5,0) circle (2pt) node[below] {$j_{n+1}$};
        \draw[->, thick] (3.2,0.25) -- (4.8,0.25);
        \draw[->, thick] (2.8,0.25) -- (1.2,0.25);
    \end{tikzpicture}
    \caption{The relevant cases for the bounds. The arrows show who needs to walk more steps, while the slash shows that a connection was not created.}
    \label{fig:cases}
\end{figure}
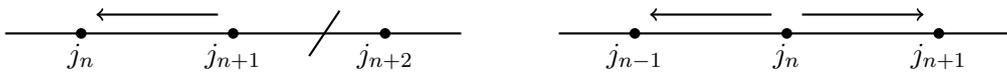

Now we can bound the probability of these events. We will show for the first case, the other cases follow similarly:
\begin{equation}\label{eq:boundingEvent}
    \begin{aligned}
        &\Prob{T^{j_{n+1}}_{S_{j_{n+1}}^{j_n}} < t^\gamma,
        S_{j_{n+1}}^{j_n} \ge \zeta_n\rb{\Delta j_n}^2, \abs{\mathfs{C}_{j_{n+1}}\rb{T^{j_{n+1}}_{S_{j_{n+1}}^{j_n}}-}}\le \Delta j_n + \Delta j_{n+1}}
        \le \\ 
        &\Prob{\sum_{i=1}^{\zeta_n\rb{\Delta j_n}^2}{t_i^{j_{n+1}}} \le t^\gamma, \abs{C_{j_n}\rb{\sum_{i=1}^{\zeta_n\rb{\Delta j_n}^2}{t_i^{j_{n+1}}}-}} \le \Delta j_n + \Delta j_{n+1}}
        \le \\ 
        &\Prob{\sum_{i=1}^{\zeta_n\rb{\Delta j_n}^2}{\rb{\Delta j_n + \Delta j_{n+1}}^\alpha \Tilde{t}_i^{j_{n+1}}} \le t^\gamma} 
        = \\ 
        &\Prob{\oov{\zeta_n\rb{\Delta j_n}^2}\sum_{i=1}^{\zeta_n\rb{\Delta j_n}^2}{\Tilde{t}_i^{j_{n+1}}} \le \frac{\rb{\Delta j_n + \Delta j_{n+1}}^{-\alpha}}{\zeta_n\rb{\Delta j_n}^2}t^{\gamma}},
    \end{aligned}
\end{equation}

using \eqref{slow<0} from Lemma \ref{lem:sizeBoundTime} in the second inequality. Now we evaluate the RHS of the last event in \eqref{eq:boundingEvent}. 
Write $\Delta j_n = d_n t$:

\begin{equation}
    \begin{aligned}
        \frac{\rb{\Delta j_n + \Delta j_{n+1}}^{-\alpha}}{\zeta_n\rb{\Delta j_n}^2}t^{\gamma}  
        &= \frac{\rb{d_n + d_{n+1}}^{-\alpha}t^{-\alpha}}{\zeta_n d_n^2 t^2}t^\gamma 
        = \frac{\rb{d_n + d_{n+1}}^{-\alpha}}{\zeta_n d_n^2}t^{\gamma-(2+\alpha)} \\
        &= \frac{\rb{d_n + d_{n+1}}^{-\alpha}}{\zeta_n d_n^2},
    \end{aligned}
\end{equation}
since $\gamma = 2+\alpha$.

We can see that $\zeta_n$ by \eqref{boundEventsA} must satisfy:
\begin{equation}
    \begin{aligned}
        1-e^{-c_g\frac{1}{\zeta_n}} &\ge 1-\ep_n\\
        e^{-c_g\frac{1}{\zeta_n}} &\le \frac{\ep}{n^2} \\
        \frac{1}{\zeta_n} &\ge -c_g\log{\frac{\ep}{n^2}} = c_g\log{\frac{n^2}{\ep}} \\
        \zeta_n &\le \frac{c_g}{\log{\frac{n^2}{\ep}}} = \frac{c_g}{2\log{n}+\log{\ep^{-1}}}
    \end{aligned}
\end{equation}
so we can take $\zeta_n = \frac{g_\ep}{\log{n}}$, where $g_\ep = \frac{c_g}{\log{\ep^{-1}}}$ is a constant small enough which only depends on $\ep$. The condition $\Delta t_i \ge 1$ is also trivial since we can consider $t>1$, and then $\Delta t_i \ge 1$. We can see that setting $d_n = g_d n$ where $g_d$ is a large enough constant is sufficient. For $\Delta_{j_i}$ we additionally require
\begin{equation}
    \begin{aligned}
        \frac{(d_n+d_{n+1})^{-\alpha}\log{(n)}}{g_\ep d_n^2} \le \frac{1}{2}.
    \end{aligned}
\end{equation}
If we plug in $d_n = g_d n$, and use $\log{n} \le g_\alpha n^{\frac{2+\alpha}{2}}$, where $g_\alpha$ is a constant which only depends on $\alpha$
\begin{equation}
    \begin{aligned}
        \frac{g_d^{-\alpha}(2n+1)^{-\alpha}\log{(n)}}{g_\ep g_d^2 n^2} \le \frac{g_\alpha g_d^{-(\alpha+2)} 2^{-\alpha}}{g_\ep}n^{-\frac{2+\alpha}{2}} \le \frac{g_\alpha g_d^{-(2+\alpha)}2^{-\alpha}k}{g_\ep} \le \frac{1}{2}
    \end{aligned}
\end{equation}
where we take $g_d = c_\alpha \log{\ep^{-1}}$ with constant $c_\alpha$ (which may depend only on $\alpha$) large enough such that this condition holds too. We also take $g_d$ to be large enough to satisfy $d_n \ge \rb{\frac{\log{n}}{g_\ep}}^2$. Note that $\alpha+2 \ge 0$, so in order for these condition to hold, we need to take $g_d$ large. Now using Bernstein's inequality:
\begin{equation}
    \begin{aligned}
        \Prob{\oov{\zeta_n\rb{\Delta j_n}^2}\sum_{i=1}^{\zeta_n\rb{\Delta j_n}^2}{\Tilde{t}_i^{j_{n+1}}} \le \frac{\rb{\Delta j_n + \Delta j_{n+1}}^{-\alpha}}{\zeta_n\rb{\Delta j_n}^2}t^{\gamma}}
        &\le
        \Prob{\oov{\zeta_n\rb{\Delta j_n}^2}\sum_{i=1}^{\zeta_n\rb{\Delta j_n}^2}{\Tilde{t}_i^{j_{n+1}}} \le \frac{1}{2}} \\
        &=
        \Prob{\oov{\zeta_n\rb{\Delta j_n}^2}\sum_{i=1}^{\zeta_n\rb{\Delta j_n}^2}{\Tilde{t}_i^{j_{n+1}}} -1 \le -\frac{1}{2}} \\
        &\le e^{-c_g \zeta_n\rb{\Delta j_n}^2} \\
        &=   e^{-c_g \frac{g_\ep}{\log{(n)}} g_d^2 n^2 t^2} \\
        &\le \frac{\ep}{n^2} = \ep_n
    \end{aligned}
\end{equation}
for $t>T$ for some constant $T$ and for all n.

Now, continuing from \eqref{inductionMain} and taking $n \to \infty$, we are left with:
\begin{equation}
    \begin{aligned}
        \Prob{G_0} &\le 10\sum_{i=0}^{\infty}{\ep_i} + 5\Prob{A,\bigcap_{i=0}^{\infty}{G_i},\bigcap_{i=0}^{\infty}{\rb{D_i}^C},B_n^\ell, G_n} + \ep \\
        &\le c_g\ep + c_g \Prob{\bigcap_{i=0}^{\infty}{G_i}}.
    \end{aligned}
\end{equation}
Since the rates are bounded from above by $M$, the process is well defined and $\Prob{\bigcap_{i=0}^{\infty}{G_i}} = 0$, which finishes the proof.

\end{proof}

We now give a construction of the process for $-2<\alpha<0$. We consider the series of processes $\cb{\mathfs{A}^M}_{M\in \BN}$. We couple the process by coupling the clocks and the random walks of each cluster. This idea is based on the process discrepancy rates of \cite{procaccia2020stationary, mu2022scaling}.

\begin{theorem}\label{thm:negative alpha limit construction}
    There exists a subsequence $M_k$ s.t $\mathfs{A}^{M_k} \to \mathfs{A}$ as $k \to \infty$ a.s on compacts in time and space. 
\end{theorem}

\begin{proof}
Fix $T>0, K>0, \ep>0$, and take $\ep_n = \frac{\ep}{n^2}$. We denote by $c_v$ any constant that depends only on $\ep,K,T,\alpha,p$ (similarly to $c_g$). By Lemma \ref{lem:M bounded size alpha<0} and translation invariance, for any choice of $\cb{j_n}_{n \in \BN}$, $\forall t<T$:
\begin{equation}
    \Prob{\abs{\mathfs{C}_{j_n}(t)} \le c_v\log{n}} \ge 1-\ep_n.
\end{equation}
Applying a union bound, this gives
\begin{equation}\label{eq:globally small size negative alpha}
    \Prob{\forall n: \abs{\mathfs{C}_{j_n}(t)} \le c_v\log{n}} \ge 1-c_g\ep.
\end{equation}
Taking $j_n = \sum_{i=1}^{n}{c_v\log{i}}$ gives that:
\begin{equation}\label{eq:globally all small size negative alpha}
    \Prob{\forall n, \forall \abs{i} \le j_n: \abs{\mathfs{C}_{i}(t)} \le c_v\log{n}} \ge 1-c_g\ep,
\end{equation}
since the differences $\Delta j_n := j_{n+1} - j_n = c_v\log{n+1}$, hence if there was a cluster $\mathfs{C}_i, j_n \le i \le j_{n+1}$, having $\abs{\mathfs{C}_i(t)} \ge \Delta j_n$ for $t<T$, it would have to be connected to either $\mathfs{C}_{j_n}$ or $\mathfs{C}_{j_{n+1}}$. By \eqref{eq:globally small size negative alpha} this happens with probability at most $\ep$. Denote the event in \eqref{eq:globally all small size negative alpha} by $G_C$.
Now, For any $M>0$ consider taking 
\begin{equation}\label{eq:K_M definition}
    i_M := \argmax_{n>0}\cb{c_v\log{n} \le M}, K_M = j_{i_M}.
\end{equation}
From this definition we can see that
\begin{equation}
    n = e^{c_vM}
\end{equation}
\begin{equation}
    K_M = \sum_{i=1}^{e^{c_vM}}{c_v\log{n}} \ge \sum_{i=1}^{c_vM}{c_vn} \ge c_vM^2.
\end{equation}

Now fix $M \le M_1 < M_2$. 
Define the left discrepancy index $D_\ell(t)$ for all $t<T$ inductively. 
We start with $D_\ell(0) = -K_M$. Then given $D_\ell(t')$ for $t' \le t$, define $\tau = \argmin_{t>0}\cb{\con{\mathfs{C}^{M_1}_{D_\ell(t)}}{\mathfs{C}^{M_1}_{D_\ell(t)+1}}{t} \ or \ \con{\mathfs{C}^{M_2}_{D_\ell(t)}}{\mathfs{C}^{M_2}_{D_\ell(t)+1}}{t}}$ that is, the minimum between the connection time of $\mathfs{C}_{D(t)}$ and $\mathfs{C}_{D(t)+1}$ in process $\mathfs{A}^{M_1}$ and between the same connection time in process $\mathfs{A}^{M_2}$. 
The upper index $\mathfs{C}^{M}_\cdot(\cdot)$ is used to indicate which process this R.V belongs to. 
We set $D_\ell(t') = D_\ell(t)$ for all $t \le t' < \tau$ and $D_\ell(\tau) = D_\ell(t)+\abs{{\mathfs{C}^{M_1}_{D_\ell(t)+1}}(t)}$.

Note that under $G_C$, the two processes $\mathfs{A}^{M_1}$ and $\mathfs{A}^{M_2}$ are identical in the interval $\bb{\max_{M' \in \cb{M_1,M_2}}\cb{X^{M'}_{D_\ell(t)}(t)+1},\min_{M' \in \cb{M_1,M_2}}\cb{X^{M'}_{D_r(t)}(t)-1}}$, by definition of the coupling. This gives that the processes are identical for $\forall t<T$ in the interval:
\begin{equation}
    \bb{\max_{t\le T, M'\in \cb{M_1,M_2}}\cb{X^{M'}_{D_\ell(t)}(t)+1},\min_{t\le T,M'\in \cb{M_1,M_2}}\cb{X^{M'}_{D_r(t)}(t)-1}}.
\end{equation}
Now Define 
\begin{equation}
    \begin{aligned}
        &L(t) := \max_{t'\le t, M'\in \cb{M_1,M_2}}\cb{X^{M'}_{D_\ell(t')}(t')+1 - D_\ell(t')} \\
        &U(t) := \min_{t'\le t, M'\in \cb{M_1,M_2}}\cb{X^{M'}_{D_r(t')}(t')-1 - D_r(t')},
    \end{aligned}
\end{equation}
and notice that $U(t),L(t)$ are Poisson point processes with rate at most $M^\alpha$.
This is because for $L(t)$ to advance, cluster $\mathfs{C}_{D_\ell(t)}$ must move in one of the processes, which happens with rate smaller than $M^\alpha$ in both processes. Since both processes are coupled using the same clocks, the combined rate is still smaller than $M^\alpha$. Once the cluster $\mathfs{C}_{D_\ell(t)}$ moves, $L(t)$ changes by 1 or 0 due to the move. It's important to notice that if the move also causes a connection from the right, the connection will not effect $L(t)$, since the change in $D_\ell(t)$ which results from the connection is exactly the same as the change in $\max_{M' \in \cb{M_1,M_2}}\cb{X^{M'}_{D_\ell(t)}(t)}$. Symmetrically the same arguments hold for $U(t)$.

By \eqref{hoeffdingBinom}, for a sufficiently small constant $c_v$
\begin{equation}\label{eq:interval M^2 steps}
    \Prob{L(0) \ge -c_vM^2} \le \ep.
\end{equation}
Now if we combine this with the bound on the rate, denoting $X_i \sim exp(1)$ i.i.d (so $M^{-\alpha}X_i \sim exp(M^\alpha)$), we get
\begin{equation}\label{eq:discrepancy slow1}
    \begin{aligned}
        \Prob{L(t) \ge -K} 
        &\le \ep + \Prob{\sum_{i=1}^{c_vM^2-K}{M^{\alpha}X_i} \le T} \\
        &\le \ep + \Prob{\frac{1}{c_vM^2}\sum_{i=1}^{c_vM^2}{X_i} \le c_vM^{-\alpha-2}} \\
        &\le 2\ep 
    \end{aligned}
\end{equation}
By WLLN, since $-(\alpha + 2) < 0$, so for big enough $M$, we have $c_vM^{-\alpha-2} \le \frac{1}{2}$.

Combining the equations, we have that
\begin{equation}
    \begin{aligned}
        \Prob{[-K,K] \ \text{is the same in both processes} \ \forall t<T} &\ge \Prob{[L(t),U(t)]\subseteq[-K,K],G_C} \\
        &\ge 1-c_g\ep.
    \end{aligned}
\end{equation}
Since for any $\ep>0$ exists $M$ large enough for this to hold for any $M_1,M_2>M$, if we take a subsequence with indexes increasing fast enough, by the Borel Cantelli lemma, this subsequence converges to a limit a.s.

\end{proof}



\section*{Acknowledgments}
The authors would like to thank Gidi Amir for introducing us to the model and questions. We would also like to thank Dominik Schmid for fruitful discussions and ideas. This project was partally supported by DFG grant 5010383.

\bibliography{ri}

\begin{thebibliography}{10}

\bibitem{berger2022growth}
N.~Berger, E.~B. Procaccia, and A.~Turner.
\newblock Growth of stationary hastings--levitov.
\newblock {\em The Annals of Applied Probability}, 32(5):3331--3360, 2022.

\bibitem{wop}
Noam Berger, Eviatar~B. Procaccia, Dominik Schmid, and Daniel Sharon.
\newblock The cluster cluster model in dimension 2 and higher.
\newblock {\em Work in Progress}.

\bibitem{DaleyJones2003PointProcesses}
D.~J. Daley and D.~Vere-Jones.
\newblock {\em An introduction to the theory of point processes. {V}ol. {I}}.
\newblock Probability and its Applications (New York). Springer-Verlag, New
  York, second edition, 2003.
\newblock Elementary theory and methods.

\bibitem{elboim2020critical}
Dor Elboim, Danny Nam, and Allan Sly.
\newblock The critical one-dimensional multi-particle dla.
\newblock {\em arXiv preprint arXiv:2009.02761}, 2020.

\bibitem{frieden2007protein}
Carl Frieden.
\newblock Protein aggregation processes: In search of the mechanism.
\newblock {\em Protein Science}, 16(11):2334--2344, 2007.

\bibitem{kesten1987long}
H.~Kesten.
\newblock How long are the arms in dla?
\newblock {\em Journal of Physics A: Mathematical and General}, 20(1):L29,
  1987.

\bibitem{lawler2010random}
Gregory~F Lawler and Vlada Limic.
\newblock {\em Random walk: a modern introduction}, volume 123.
\newblock Cambridge University Press, 2010.

\bibitem{leyvraz2003scaling}
Francois Leyvraz.
\newblock Scaling theory and exactly solved models in the kinetics of
  irreversible aggregation.
\newblock {\em Physics Reports}, 383(2-3):95--212, 2003.

\bibitem{losev2025long}
Ilya Losev and Stanislav Smirnov.
\newblock How long are the arms in dbm?
\newblock {\em Communications in Mathematical Physics}, 406(4):1--16, 2025.

\bibitem{lushnikov2006gelation}
Alex~A Lushnikov.
\newblock Gelation in coagulating systems.
\newblock {\em Physica D: Nonlinear Phenomena}, 222(1-2):37--53, 2006.

\bibitem{lushnikov1978coagulation}
Alexei~A Lushnikov.
\newblock Coagulation in finite systems.
\newblock {\em Journal of Colloid and interface science}, 65(2):276--285, 1978.

\bibitem{meakin1984diffusion}
Paul Meakin.
\newblock Diffusion-limited aggregation in three dimensions: results from a new
  cluster-cluster aggregation model.
\newblock {\em Journal of colloid and interface science}, 102(2):491--504,
  1984.

\bibitem{meakin1984effects}
Paul Meakin.
\newblock Effects of cluster trajectories on cluster-cluster aggregation: A
  comparison of linear and brownian trajectories in two-and three-dimensional
  simulations.
\newblock {\em Physical Review A}, 29(2):997, 1984.

\bibitem{meakin1985dynamic}
Paul Meakin, Tam{\'a}s Vicsek, and Fereydoon Family.
\newblock Dynamic cluster-size distribution in cluster-cluster aggregation:
  Effects of cluster diffusivity.
\newblock {\em Physical Review B}, 31(1):564, 1985.

\bibitem{mu2022scaling}
Y.~Mu, E.~B. Procaccia, and Y.~Zhang.
\newblock Scaling limit of dla on a long line segment.
\newblock {\em Transactions of the American Mathematical Society}, 2022.

\bibitem{ostwald1903lehrbuch}
Wilhelm Ostwald.
\newblock {\em Lehrbuch der allgemeinen Chemie}, volume~1.
\newblock W. Engelmann, 1903.

\bibitem{procaccia2020stationary}
E.~B. Procaccia, J.~Ye, and Y.~Zhang.
\newblock Stationary dla is well defined.
\newblock {\em Journal of Statistical Physics}, 181:1089--1111, 2020.

\bibitem{procaccia2021dimension}
Eviatar~B Procaccia and Itamar Procaccia.
\newblock Dimension of diffusion-limited aggregates grown on a line.
\newblock {\em Physical Review E}, 103(2):L020101, 2021.

\bibitem{rajesh2024exact}
R~Rajesh, V~Subashri, and Oleg Zaboronski.
\newblock Exact calculation of the probabilities of rare events in
  cluster-cluster aggregation.
\newblock {\em Physical Review Letters}, 133(9):097101, 2024.

\bibitem{rath2009erdHos}
Balazs Rath and Balint Toth.
\newblock Erd{\H{o}}s-renyi random graphs forest fires self-organized
  criticality.
\newblock {\em Electronic Journal of Probability [electronic only]},
  14:1290--1327, 2009.

\bibitem{sidoravicius2017one}
Vladas Sidoravicius and Balazs Rath.
\newblock One-dimensional multi-particle dla--a pde approach.
\newblock {\em arXiv preprint arXiv:1709.00484}, 2017.

\bibitem{sly2020one}
Allan Sly.
\newblock On one-dimensional multi-particle diffusion limited aggregation.
\newblock In {\em In and Out of Equilibrium 3: Celebrating Vladas
  Sidoravicius}, pages 755--774. Springer, 2020.

\bibitem{smoluchowski1918versuch}
M~v Smoluchowski.
\newblock Versuch einer mathematischen theorie der koagulationskinetik
  kolloider l{\"o}sungen.
\newblock {\em Zeitschrift f{\"u}r physikalische Chemie}, 92(1):129--168, 1918.

\bibitem{voorhees1985theory}
Peter~W Voorhees.
\newblock The theory of ostwald ripening.
\newblock {\em Journal of Statistical Physics}, 38:231--252, 1985.

\end{thebibliography}
\bibliographystyle{plain}

%
%



\end{document}